\documentclass[11pt]{article}
\makeatletter

\usepackage{amssymb,amsmath,amsthm,latexsym}
\title{\bf Maximum scattered linear sets of pseudoregulus type and the Segre Variety ${\cal S}_{n,n}$}
\author{G. Lunardon, G. Marino, O. Polverino, R. Trombetti}
\date{}

\begin{document}
\maketitle

\newtheorem{theorem}{Theorem}[section]
\newtheorem{lemma}[theorem]{Lemma}
\newtheorem{conj}[theorem]{Conjecture}
\newtheorem{remark}[theorem]{Remark}
\newtheorem{cor}[theorem]{Corollary}
\newtheorem{prop}[theorem]{Proposition}
\newtheorem{defin}[theorem]{Definition}
\newtheorem{result}[theorem]{Result}
\newtheorem{property}[theorem]{Property}
\newtheorem{example}[theorem]{Example}

\makeatother
\newcommand{\Prf}{\noindent{\bf Proof}.\quad }
\renewcommand{\labelenumi}{(\alph{enumi})}


\def\B{\mathbf B}
\def\C{\mathbf C}
\def\Z{\mathbf Z}
\def\Q{\mathbf Q}
\def\W{\mathbf W}
\def\a{\mathbf a}
\def\b{\mathbf b}
\def\c{\mathbf c}
\def\d{\mathbf d}
\def\e{\mathbf e}
\def\l{\mathbf l}
\def\v{\mathbf v}
\def\w{\mathbf w}
\def\x{\mathbf x}
\def\y{\mathbf y}
\def\z{\mathbf z}
\def\t{\mathbf t}
\def\cD{\mathcal D}
\def\cC{\mathcal C}
\def\cH{\mathcal H}
\def\cM{{\mathcal M}}
\def\cK{\mathcal K}
\def\cQ{\mathcal Q}
\def\cU{\mathcal U}
\def\cS{\mathcal S}
\def\cT{\mathcal T}
\def\cR{\mathcal R}
\def\cN{\mathcal N}
\def\cA{\mathcal A}
\def\cF{\mathcal F}
\def\cL{\mathcal L}
\def\cP{\mathcal P}
\def\cI{\mathcal I}

\def\Sq{{\mathcal S}_{n,n}(q)}
\def\Sqn{{\mathcal S}_{n,n}(q^n)}
\def\Oq{\Omega({\mathcal S}_{n,n})(q)}
\def\Oqn{\Omega({\mathcal S}_{n,n})(q^n)}

\def\PG{{\rm PG}}
\def\GF{{\rm GF}}

\def\Pg{PG(5,q)}
\def\pg{PG(3,q^2)}
\def\ppg{PG(3,q)}
\def\HH{{\cal H}(2,q^2)}
\def\F{\mathbb F}
\def\Ft{\mathbb F_{q^t}}
\def\P{\mathbb P}
\def\V{\mathbb V}
\def\bS{\mathbb S}
\def\G{\mathbb G}
\def\E{\mathbb E}
\def\N{\mathbb N}
\def\K{\mathbb K}
\def\C{\mathbb C}
\def\M{\mathbb M}
\def\ps@headings{
 \def\@oddhead{\footnotesize\rm\hfill\runningheadodd\hfill\thepage}
 \def\@evenhead{\footnotesize\rm\thepage\hfill\runningheadeven\hfill}
 \def\@oddfoot{}
 \def\@evenfoot{\@oddfoot}
}

\def\E{\mathbb E}
\def\P{\mathbb P}
\def\cD{\mathcal D}
\def\cI{\mathcal I}
\def\cO{\mathcal O}
\def\cG{\mathcal G}

\def\Sq{{\mathcal S}_{n,n}(q)}
\def\Sqn{{\mathcal S}_{n,n}(q^n)}
\def\Oq{\Omega({\mathcal S}_{n,n})(q)}
\def\Oqn{\Omega({\mathcal S}_{n,n})(q^n)}

\def\uv{\underbar v}
\def\uw{\underbar w}

\begin{abstract}
In this paper we study a family of scattered $\F_q$--linear sets of
rank $tn$ of the projective space $PG(2n-1,q^t)$  ($n \geq 1$,
$t\geq 3$), called of {\it pseudoregulus type}, generalizing results
contained in \cite{LV} and in \cite{MPT}. As an application, we
characterize, in terms of the associated linear sets, some classical
families of semifields: the Generalized Twisted Fields and the
2-dimensional Knuth semifields.

\end{abstract}

\bigskip

\par\noindent
{\bf Keywords:}  linear set, subgeometry, semifield

\section{Introduction}
In recent years the theory of linear sets has constantly increased
its importance mainly because of its connection with other geometric
objects such as blocking sets, translation ovoids and semifield
planes (for an overview see \cite{Polverino2010}).

In this paper we study a family of maximum scattered $\mathbb
F_q$-linear sets  of the projective space $\Lambda=PG(2n-1,q^t)$
($n\geq 1$, $ t\geq 3$). They were first introduced in \cite{MPT}
for $n=2$ and $t=3$, and further generalized in \cite{LV} for
$n\geq 2$ and $t=3$. If $\Lambda$ is not a line, it is possible to
associate with any such linear set a family of
$(q^{nt}-1)/(q^t-1)$ pairwise disjoint lines admitting exactly two
$(n-1)$-dimensional transversal spaces. Such a set of lines is
called {\it pseudoregulus}, in analogy to the pseudoregulus of
$PG(3,q^2)$ introduced by Freeman in \cite{Fr}. For this reason,
we refer to the relevant family of linear sets as {\it linear sets
of pseudoregulus type}.

All maximum scattered $\mathbb F_q$-linear sets of
$\Lambda=PG(2n-1,q^3)$ ($n\geq 2$) are of pseudoregulus type and
they are all equivalent under the action of the collineation group
of $\Lambda$ (see \cite[Propositions 2.7 and 2.8]{MPT} for $n=2$ and
\cite[Section 3 and Theorem 4]{LV} for $n\geq 3$). In this paper, we
characterize $\mathbb F_q$--linear sets of $PG(2n-1,q^t)$ ($n\geq
1$, $t\geq 3$) of pseudoregulus type in terms of the associated
projected subgeometry and we prove that, when $n>1$ there are $\varphi(t)/2$
(where $\varphi$ denotes the Euler's phi function) orbits of such
$\mathbb F_q$-linear sets under the action of the collineation group
of $PG(2n-1,q^t)$ (Theorems \ref{theorem:DesargProj},
\ref{thm:numbprojequipseudoregulus}). Also, we show that,
when $t\geq 4$ and $q>3$, there exist examples of maximum scattered
$\F_q$--linear sets of $PG(2n-1,q^t)$ ($n\geq 1$) which are not of
pseudoregulus type (Example \ref{example:noPseud}).

Finally, in Section \ref{sec:SegreVariety} we first prove some
geometric properties of the Segre Variety $\cS_{n,n}$ of the
projective space $\P=PG(n^2-1,q)$. These properties, together with
the results contained in Sections \ref{subsec:scattered} and
\ref{sec:scattered-line}, allow us to describe and characterize the
linear sets associated with some classical semifields: the
Generalized Twisted Fields and the Knuth semifields 2-dimensional
over their left nucleus (Propositions \ref{prop:linsetGenTwFi} and
\ref{prop:knuthlinearset}, Theorems \ref{theorem:caraGenTwFie} and
\ref{theorem:caraKnuth}).

\section{Preliminary results}\label{sec:preliminary}
A $(t-1)$--spread of a projective space $PG(nt-1, q)$ is a family
$\cS$ of mutually disjoint subspaces of dimension $t-1$ such that
each point of $PG(nt-1,q)$ belongs to an element of $\cS$.\\
A first example of spread can be obtained in the following way. Let
$PG(n-1,q^t)=PG(V,\F_{q^t})$. Any point $P$ of $PG(n-1,q^t)$ defines
a $(t-1)$--dimensional subspace $X(P)$ of the projective space
$PG(nt-1,q)=PG(V,\F_q)$ and $\cD=\{X(P):P\in PG(n-1,q^t)\}$ is a
spread of $PG(nt-1,q)$, called a {\it Desarguesian spread} (see
\cite{Segre}, Section 25)$^($\footnote{In \cite{Segre} a
Desarguesian spread is called "Sistema Grafico Elementare".}$^)$. If
$n>2$, the incidence structure $\Pi_{n-1}(\cD)$, whose points are
the elements of $\cD$ and whose lines are the $(2t-1)$--dimensional
subspaces of $PG(nt-1,q)$ joining two distinct elements of $\cD$, is
isomorphic to $PG(n-1,q^t)$. The structure $\Pi_{n-1}(\cD)$ is
called the $\F_q$--{\em linear representation} of $PG(n-1,q^t)$.

A Desarguesian $(t-1)$--spread of $PG(nt-1,q)$ can also be obtained
 as follows  (see \cite[Section 27]{Segre}, \cite{Lu1999} and
\cite{BaLu2011}). Embed $\Sigma\simeq PG(nt-1,q)$ in $\Sigma^* =
PG(nt-1,q^t)$ in such a way that $\Sigma$ is the set of fixed points
of a semilinear collineation $\Psi$ of $\Sigma^*$ of order $t$. Let
$\Theta= PG(n-1,q^t)$ be a subspace of $\Sigma^*$ such that
$\Theta,$ $\Theta^{\Psi},...,$ $\Theta^{\Psi^{t-1}}$ span the whole
space $\Sigma^*.$ If $P$ is a point of $\Theta$, $X^*(P)=\langle
P,P^{\Psi},...,P^{\Psi^{t-1}}\rangle_{q^t}$ is a
$(t-1)$--dimensional subspace of $\Sigma^*$ defining a
$(t-1)$--dimensional subspace $X(P) = X^*(P) \cap \Sigma$ of
$\Sigma$. As $P$ varies over the subspace $\Theta$ we get a set of
$q^{t(n-1)}+q^{t(n-2)}+\dots+q^t+1$ mutually disjoint
$(t-1)$--dimensional subspaces of $\Sigma.$ Such a set is denoted by
$\cD=\cD(\Theta)$ and it turns out to be a Desarguesian
$(t-1)$--spread of $\Sigma$. The $(n-1)$--dimensional subspaces
$\Theta,$ $\Theta^{\Psi},...,$ $\Theta^{\Psi^{t-1}}$  are uniquely
defined by the Desarguesian spread $\cD$, i.e. $\cD(\Theta)=\cD(X)$
if and only if $X=\Theta^{\Psi^i}$ for some $i\in \{0,1,\dots,t-1\}$
and, following the terminology used by Segre \cite[page 29]{Segre},
we will refer to them as {\it director spaces} of $\cD$ (see  also
\cite[Theorem 3]{LV}).

\begin{remark}\label{rema:CharDesSpread} \rm
Let $\cS$ be a $(t-1)$-spread of $\Sigma = PG(nt-1,q)$ embedded in
$\Sigma^* = PG(nt-1,q^t)$ in such a way that $\Sigma=Fix(\Psi)$
where $\Psi$ is a semilinear collineation  of $\Sigma^*$ of order
$t$. If  $H$ is an $(n-1)$-dimensional subspace of $\Sigma^*$ such
that\\
 (i) $\Sigma^*=\langle H, H^\Psi,\dots,H^{\Psi^{t-1}}\rangle_{q^t}$;\\
 (ii) $X^*\cap H\neq \emptyset$ for each
$(t-1)$-dimensional
subspace $X^*$ of $\Sigma^*$ such that $X^*\cap \Sigma \in \cS$;\\
then it is easy to see that $\cD(H)=\cS$, i.e. $\cS$ is a
Desarguesian spread and $H$ is one of its director spaces.
\end{remark}

\subsection{Linear sets}\label{sec:linearsets} Let
$\Lambda=PG(r-1,q^t)=PG(V,\F_{q^t})$, $q=p^h$, $p$ prime, and let
$L$ be a set of points of $\Lambda$. The set $L$ is said to be an
{\em $\F_q$--linear} set of $\Lambda$ if it is defined by the
non--zero vectors of an $\F_q$--vector subspace $U$ of $V$, i.e., $
L=L_U=\{\langle {\bf u}\rangle_{q^t}: {\bf u}\in U\setminus\{{\bf
0}\}\}$. If $dim _{\F_q} U=k$, we say that $L$ has {\it rank} $k$.
If $\Omega=PG(W,\F_{q^t})$ is a subspace of $\Lambda$ and $L_U$ is
an $\F_q$-linear set of $\Lambda$, then $\Omega\cap L_U$ is an
$\F_q$--linear set of $\Omega$ defined by the $\F_q$--vector
subspace $U\cap W$, and we say that $\Omega$ has weight $i$ in $L_U$
if $dim_{\F_q}(W\cap U)=i$ and we write $\omega_{L_U}(\Omega)=i$. If
$L_U\neq \emptyset$, we have
\begin{eqnarray}
&&\hskip -.8cm |L_U|\leq  q^{k-1}+q^{k-2}+\dots+q+1,\label{form4}\\
&&\hskip -.8cm |L_U |\equiv 1\, (mod \, q)\label{form3}.
\end{eqnarray}

\noindent
 For further details on linear sets see
\cite{Polverino2010}.

\medskip

An $\F_q$--linear set $L_U$ of $\Lambda$ of rank $k$ is {\em
scattered} if all of its points have weight 1, or equivalently, if
$L_U$ has maximum size $q^{k-1}+q^{k-2}+\cdots+q+1$.

In \cite{BL}, the authors prove the following result on scattered
linear sets.
    \begin{theorem}{\rm\cite[Theorem 4.2]{BL}}\label{Thm. BL}
A scattered  $\F_q$--linear set of $PG(r-1,q^t)$ has rank at most
rt/2. \end{theorem}

\noindent A scattered $\F_q$--linear set $L$ of $PG(r-1,q^t)$ of
maximum rank $rt/2$ is called {\em maximum scattered} linear set.

\begin{remark}\label{rema:sclinset} \rm
Note that if $L_U$ is a scattered $\F_q$-linear set of $PG(r-1,q^t)$
of rank $k$ containing more than one point, then by
$|L_U|=q^{k-1}+q^{k-2}+\cdots+q+1$ and (\ref{form3}), $L_U$ is not
an $\F_{q^s}$-linear set for each subfield $\F_{q^s}$ of $\F_{q^t}$
properly containing $\F_q$. In other words, a scattered
$\F_q$--linear set $L$ of rank $k>1$ of $PG(r-1,q^t)$ is not a
linear set of rank $n<k$. Also, by Theorem \ref{Thm. BL}, a maximum
scattered linear set of $PG(r-1,q^t)$ spans the whole space.
\end{remark}

If $dim_{\F_q}U=dim_{\F_{q^t}}V=r$ and $\langle U \rangle_{q^t}=V$,
then  the $\F_q$--linear set $L_U$ is a {\it subgeometry} of
$PG(V,\F_{q^t})=PG(r-1,q^t)$ isomorphic to $PG(r-1,q).$ If $t=2$,
then $L_U$ is a Baer subgeometry of $PG(r-1,q^2)$.

In \cite{LP}, the authors give the following characterization of
$\F_q$--linear sets. Let $\Sigma=PG(k-1,q)$ be a subgeometry of
$\Sigma^*=PG(k-1,q^t)$, let $\Gamma$ be a $(k-r-1)$--dimensional
subspace  of $\Sigma^*$ disjoint from $\Sigma$ and let
$\Lambda=PG(r-1,q^t)$ be an $(r-1)$--dimensional subspace of
$\Sigma^*$ disjoint from $\Gamma$. Denote by
$$L=\{\langle \Gamma,P \rangle_{q^t} \cap \Lambda  \,: \, P\in \Sigma \}$$
the projection  of $\Sigma$ from $\Gamma$ to $\Lambda.$ We call
$\Gamma$ and $\Lambda$, respectively, the $center$ and the $axis$ of
the projection. Denote by $p_{\Gamma,\Lambda}$ the map from $\Sigma$
to $L$ defined by $P \mapsto \langle \Gamma,P\rangle_{q^t} \cap
\Lambda$ for each point $P$ of $\Sigma.$ By definition
$p_{\Gamma,\Lambda}$ is surjective and
$L=p_{\Gamma,\Lambda}(\Sigma).$

\begin{theorem}{\rm\cite[Theorems 1 and 2]{LP}}\label{Theor LP1}
If $L$ is a projection of $\Sigma=PG(k-1,q)$ to
$\Lambda=PG(r-1,q^t),$ then $L$ is an $\F_q$--linear set of
$\Lambda$ of rank $k$ and $\langle L \rangle_{q^t}=\Lambda$.
Conversely, if $L$ is an $\F_q$--linear set of $\Lambda$ of rank $k$
and $\langle L \rangle_{q^t}=\Lambda$, then either $L$ is a
subgeometry of $\Lambda$ or for each $(k-r-1)$--dimensional subspace
$\Gamma$ of $\Sigma^*=PG(k-1,q^t)$ disjoint from $\Lambda$ there
exists a subgeometry $\Sigma$ of $\Sigma^*$ disjoint from $\Gamma$
such that $L=p_{\Gamma,\Lambda}(\Sigma).$
\end{theorem}
Also, in \cite{LaVanVo2010} it has been proven:
\begin{theorem}{\rm\cite[Theorem 3]{LaVanVo2010}}\label{thm:LavVanEquivalLinearSet}
Let $L_1=p_{\Gamma_1,\Lambda}(\Sigma_1)$ and
$L_2=p_{\Gamma_2,\Lambda}(\Sigma_2)$ be two $\F_q$--linear sets of
rank $k$ of $\Lambda=\langle L_1\rangle_{q^t}=\langle
L_2\rangle_{q^t}$, and suppose that $L_i$ is not a linear set of
rank $n<k$. Then $L_1$ and $L_2$ are projectively equivalent if and
only if there exists $\beta\in Aut(\Sigma^*)$ such that
$\Sigma_1^\beta=\Sigma_2$ and $\Gamma_1^\beta=\Gamma_2$.
\end{theorem}

\begin{remark}\label{rem:axis}
{\rm Note that, if $S_\Gamma=\Sigma^*/\Gamma \simeq PG(r-1,q^t)$
denotes the $(r-1)$--dimensional space obtained as quotient geometry
of $\Sigma^*$ over $\Gamma$, then the set $L_{\Gamma,\Sigma}$ of the
$(k-r)$-dimensional subspaces of $\Sigma^*$ containing $\Gamma$ and
with non-empty intersection with $\Sigma$ is an $\F_q$-linear set of
the space $S_{\Gamma}$ isomorphic to
$L=p_{\Gamma,\Lambda}({\Sigma})$, for each $(r-1)$--dimensional
space $\Lambda$ disjoint from $\Gamma$. This means that
$p_{\Gamma,\Lambda}({\Sigma})$ is isomorphic to the $\F_q$--linear
set $\{P+\Gamma:\ P\in\Sigma\}$ of the quotient space $S_\Gamma$,
and hence it does not depend on the choice of the axis $\Lambda$,
and we will simply denote it as $p_{\Gamma}(\Sigma)$.}
\end{remark}

\section{Maximum scattered $\F_q$--linear sets of
pseudoregulus  type in $PG(2n-1,q^t)$}\label{subsec:scattered}

In this section we study a family of maximum scattered linear sets
to which a geometric structure, called {\it pseudoregulus}, can be
associated. This generalizes results contained in \cite{MPT},
\cite{LMPT} and \cite{LV}.

\begin{defin} \label{def:pseud} \rm
Let $L=L_U$ be a scattered $\F_q$--linear set of
$\Lambda=PG(2n-1,q^t)$ of rank $tn$, $t,n \geq 2$. We say that $L$
is of {\it
pseudoregulus type} if\\
(i) there exist $m=\frac{q^{nt}-1}{q^t-1}$ pairwise disjoint lines
of $\Lambda$, say $s_1,s_2,\dots,s_m$, such that
$$w_L(s_i)=t, \,\, i.e.\,\,  |L\cap s_i|=q^{t-1}+q^{t-2}+\cdots+q+1 \ \ \forall \,\, i=1,\dots,m;$$
(ii) there exist exactly two  $(n-1)$--dimensional subspaces $T_1$
and $T_2$ of $\Lambda$ disjoint from $L$ such that $T_j\cap
s_i\neq \emptyset$ for each  $i=1,\dots,m$ and $j=1,2$.
\end{defin}

\medskip

\noindent We call the set of lines $\mathcal{P}_L = \{s_i \colon
i=1,\dots,m\}$ the {\it $\F_q$--pseudoregulus} (or simply {\it
pseudoregulus}) of $\Lambda$ associated with $L$ and we refer to
$T_1$ and $T_2$ as {\it transversal spaces} of $\mathcal{P}_L$ (or
transversal spaces of $L$). Note that, by Remark
\ref{rema:sclinset}, $L$ spans the whole space and hence the
transversal spaces $T_1$ and $T_2$ are disjoint. When $t=n=2$, these
objects already appeared in \cite{Fr},  where the term
\`{\`{}}pseudoregulus\'{\'{}} was introduced for the first time.
\bigskip

If $L$ is a scattered $\F_q$-linear set of the projective space
$PG(r-1,q^t)$, then by Theorem \ref{Thm. BL} every $h$-dimensional
subspace $X_h$ of $PG(r-1,q^t)$ intersects $L$ in a linear set of
rank at most $(h+1)t/2$, i.e. the weight of $X_h$ in $L$ is at
most $(h+1)t/2$. So we get

\begin{prop} \label{prop:rette disgiunte}
If $t\geq 3$, then the lines of weight \,$t$ in a scattered
$\F_q$-linear set $L$ of $PG(r-1,q^t)$ of rank $h$ are pairwise
disjoint and hence the number of such lines is at most
$(q^h-1)/(q^t-1)$.
\end{prop}
\begin{proof}
If $\ell$ and $\ell'$ are distinct lines of $PG(r-1,q^t)$ of weight
$t$ in $L$ and $\ell\cap \ell'\neq \emptyset$, then the plane $\pi$
joining $\ell$ and $\ell'$ has weight  at least $2t-1$ in $L$. On
the other hand, since $\pi \cap L$ is a scattered $\F_q$-linear set
of the plane $\pi$, by Theorem \ref{Thm. BL}, we also have that the
weight of $\pi$ in $L$ is at most $3t/2$; so we get $t\leq 2$, a
contradiction. Hence, the number of lines of $PG(r-1,q^t)$ having
weight $t$ in $L$ is at most
$$\frac{q^{h-1}+q^{h-2}+\cdots+q+1}{q^{t-1}+q^{t-2}+\cdots+q+1}=\frac{q^h-1}{q^t-1}.$$
\end{proof}

\noindent As a consequence of Proposition \ref{prop:rette disgiunte}
we get

\begin{cor} \label{cor:unicPseudo}
If $L$ is an $\F_q$-linear set of pseudoregulus type of the
projective space $\Lambda=PG(2n-1,q^t)$, with  $t\geq 3$,  then the
associated pseudoregulus is the set of all the lines of $\Lambda$ of
weight $t$ in $L$. Hence, the pseudoregulus associated with $L$ and
its transversal spaces are uniquely determined.
\end{cor}

\begin{remark} \rm  \label{rema:Subgeometry}
If $t=2$, a scattered $\F_q$-linear set $L$ of rank $2n$ of the
projective space $\Lambda$ is a Baer subgeometry   isomorphic to
$PG(2n-1,q)$ and each line spread of $L$ produces a set of lines of
$\Lambda$ satisfying (i) of Definition \ref{def:pseud} and each
Desarguesian line spread, say $\cD$,  of $L$ gives a set of lines of
$\Lambda$ satisfying both (i) and (ii) of Definition
\ref{def:pseud}. In this last case the transversal spaces are the
two $(n-1)$-dimensional director spaces of the  Desarguesian spread
$\cD$. So each maximum scattered $\F_q$-linear set $L$ of the
projective space $\Lambda=PG(2n-1,q^2)$ is of pseudoregulus type,
but in this case the associate pseudoregulus is not uniquely
defined; also, since these linear sets are Baer subgeometries, they
are all projectively equivalent. The same happens in the next case,
i.e. each maximum scattered $\F_q$-linear set $L$ of the projective
space $\Lambda=PG(2n-1,q^3)$ ($n\geq 2$) is of pseudoregulus type
and they are all projectively equivalent (see \cite{MPT} and
\cite[Theorem 4, Lemma 5, Lemma 7,
 Theorem 10]{LV}). Whereas, when $t>3$: (i) there exist maximum
scattered $\F_q$-linear sets of the projective space
$\Lambda=PG(2n-1,q^t)$ which are not of pseudoregulus type (see
Example \ref{example:noPseud}); (ii) $\F_q$-linear sets of
pseudoregulus type, in general, are not all projectively
equivalent (see Theorem \ref{thm:numbprojequipseudoregulus}).
\end{remark}

\medskip

 \noindent The construction presented in \cite[Section
2]{LMPT} when  $t=3$ and $n=2$, can be generalized providing a
simple way to construct scattered $\F_q$-linear sets of
pseudoregulus type of $PG(2n-1,q^t)$ for any $t,n\geq 2$.\\
\noindent Let $\Lambda=PG(V,\F_{q^t}),$ where $V=V(2n,\F_{q^t}) =
U_1 \oplus U_2$, with $dim\,U_1=dim\,U_2=n$ and let
$T_1=PG(U_1,\F_{q^t})$ and $T_2=PG(U_2,\F_{q^t})$. Now, let $\Phi_f$
be a semilinear collineation  between $T_1$ and $T_2$, induced by
the invertible semilinear map  $f : U_1 \rightarrow U_2$, having as
companion automorphism an element $\sigma\in Aut(\F_{q^t})$ such
that $Fix(\sigma)=\F_q$. Then for each $\rho \in \F_{q^t}^*$, the
set
$$W_{\rho, f} = \{\underbar{u}+\rho f(\underbar{u}) \,:\, \underbar{u}\in U_1\}$$
 is an $\F_q$--vector subspace of $V$ of dimension $tn$ and it is not difficult to see that
$L=L_{W_{\rho,f}}$ is an $\F_q$-linear set of $\Lambda$ or rank $tn$
of scattered type $^($\footnote{More generally, if $Fix
(\sigma)=\F_{q'}$, then $L_{W_{\rho,f}}$ is an $\F_{q'}$-linear set
of $\Lambda$ of scattered type.}$^)$. Also, we can see that for each
line $s_{P}$ joining the points $P=\langle \underbar u
\rangle_{q^t}$ and $P^{\Phi_f}=\langle f(\underbar u) \rangle_{q^t}$
of $T_1$ and $T_2$ respectively, we have that
\begin{equation}\label{generalizedpseudoregulus}s_{P} \cap L=\{\langle \lambda \underbar u+\lambda^\sigma f(\underbar u) \rangle_{q^t}\, :\,
\lambda \in \F_{q^t}^* \}.\end{equation}
 Hence, the line $s_{P}$, for each $P\in T_1$,
 has weight $t$ in
$L$. Also, if  $P \neq Q$, the lines $s_{P}$ and $s_{Q}$ are
disjoint. This means that  $L$ satisfies (i) of Definition
\ref{def:pseud}. Moreover, it is clear that $T_1 \cap s_P=\{P\}$ and
$T_2 \cap s_P=\{P^{\Phi_f}\}$ for each $P\in T_1$ and that $T_1\cap
L=T_2\cap L=\emptyset$. In addition, $T_1$ and $T_2$ are the only
$(n-1)$-dimensional transversal spaces of the lines $s_P$. Indeed,
if $T=PG(U,\F_{q^t})=PG(n-1,q^t)$ were another transversal space,
then $T$ would be disjoint from $T_1$ and $T_2$ and, since $T\cap
s_{P}\neq \emptyset$ for each $P\in T_1$, we have that
$$U=\{\underbar u+\lambda_{\underbar u}f(\underbar u)\,:\, \underbar
u\in U_1\},$$
 where $\lambda_{\underbar u} \in \F_{q^t} \,\,$ for each $\underbar u \in U_1$ and $\lambda_{\underbar u} \neq 0$
  for each ${\underbar u} \neq \underbar 0$. Now, since $U$ is an $\Ft$-subspace of
$V$, the map $f$ turns out to be an $\F_{q^t}$-linear map of $V$,
a contradiction. So, also $L$ satisfies  (ii) of Definition
\ref{def:pseud} and hence $L_{W_{\rho, f}}$ is a maximum scattered
$\F_q$-linear set of $\Lambda$ of  pseudoregulus type and
$\mathcal{P}_{L}=\{s_{P} \,:\, P\in T_1\}$  is its associated
pseudoregulus. Hence we have proved the following

\begin{theorem}\label{thm:algebraicpseudoregulus}
Let $T_1=PG(U_1,\F_{q^t})$ and $T_2=PG(U_2,\F_{q^t})$ be two
disjoint $(n-1)$-dimensional subspaces of
$\Lambda=PG(V,\F_{q^t})=PG(2n-1,q^t)$ ($t>1$) and let $\Phi_f$ be a semilinear collineation between $T_1$ and $T_2$  having as companion
automorphism an element $\sigma\in Aut(\F_{q^t})$ such that
$Fix(\sigma)=\F_q$. Then, for each $\rho \in \F_{q^t}^*$, the set
$$L_{\rho,f} = \{\langle \underbar{u}+\rho f(\underbar{u}) \rangle_{q^t} \,:\, \underbar{u}\in U_1\setminus\{\underbar 0\}\}$$
is an $\F_q$-linear set of $\Lambda$ of pseudoregulus type whose
associated pseudoregulus is $\mathcal{P}_{L_{\rho,f}}=\{\langle P,
P^{\Phi_f}\rangle_{q^t}\,:\, P\in T_1\}$, with transversal spaces
$T_1$ and $T_2$.

\end{theorem}

\begin{remark}\rm
Note that, with the notation of the previous theorem, if $L_{\rho,f}
\cap L_{\rho',f} \neq \emptyset$, then $L_{\rho,f}=L_{\rho',f}$ and
this happens if and only if $N_{q^t/q}(\rho)= N_{q^t/q}(\rho')$
$^(\footnote{Here $N_{q^t/q}(\cdot)$ denotes the norm function from
$\F_{q^t}$ on $\F_q$.}^)$. Hence $T_1$, $T_2$ and the collineation
$\Phi_f$ define a set of $q-1$ mutually disjoint linear sets of
pseudoregulus type admitting the same associated pseudoregulus
$\mathcal{P}$  and covering, together with the transversal spaces
$T_1$ and $T_2$, the point set of $\mathcal{P}$.
\end{remark}

\bigskip

Up to projective equivalence, the scattered $\F_q$-linear sets
$L_{\rho,f}$ only depend on the field automorphism associated with
$f$. Indeed, we have

\begin{theorem} \label{theorem:ProjEquiv}
The $\F_q$-linear sets of $\Lambda=PG(2n-1,q^t)$ ($n\geq 2, t\geq
2$) $L_{\rho,f}$ and $L_{\rho',g}$ are $P\Gamma L$-equivalent if
and only if $\sigma_f=\sigma_g^{\pm 1}$ where $\sigma_f$ and
$\sigma_g$ are the automorphisms associated with $f$ and $g$,
respectively.
\end{theorem}

\begin{proof} If $t=2$, the assertion follows from Remark \ref{rema:Subgeometry}. Let $t>2$. Then, by Corollary \ref{cor:unicPseudo} the transversal spaces associated with $L_{\rho,f}$ and
$L_{\rho',g}$ are uniquely determined. Hence, up to the action of
$PGL(2n,q^t)$, we may assume that the transversal spaces of
$L_{\rho,f}$ and $L_{\rho',g}$ are the same. Also, since
$L_{\rho,f}=L_{1,\rho^{-1} f}$, we may consider scattered
$\F_q$-linear sets of the form $L_{1,f}$. Suppose that $L_{1,f}$ and
$L_{1,g}$ are projectively equivalent; i.e., there exists a
collineation $\phi_F$ of $\Lambda = PG(2n-1,q^t)=PG(V,\F_{q^t})$
defined by an invertible semilinear map $F$ of the vector space $V$
having companion automorphism $\tau$, such that
$\phi_F(L_{1,f})=L_{1,g}$. By Corollary \ref{cor:unicPseudo}
$\phi_F(\{T_1,T_2\})=\{T_1,T_2\}$. Precisely, either
$\phi_F(T_i)=T_i,$ $i=1,2$ or $\phi_F(T_i)=T_j,$ $\{i,j\}=\{1,2\}$.

 In the first case, we have that $F(U_1)=U_1$
and $F(U_2)=U_2$. Since $\phi_F(L_{1,f})=L_{1,g}$, for each
$\underbar{u} \in U_1$ we have
 $\phi_F(\langle \underbar{u}+f(\underbar{u}) \rangle_{q^t})=\langle F(\underbar{u}+f(\underbar{u})) \rangle_{q^t} \in L_{1,g};$ in other words, for each vector $\underbar{u} \in U_1$ we have
  \begin{equation}\label{eq:colloineationdescribed}F(\underbar{u}+f(\underbar{u}))=\lambda_{\underbar{u}}(\underbar{u}'+g(\underbar{u}'))=\lambda_{\underbar{u}}\underbar{u}' + \lambda_{\underbar{u}}g(\underbar{u}'),\end{equation} where
  $\underbar{u}' \in U_1$ and $\lambda_{\underbar{u}}\in \F_{q^t}^*$ if $\underbar{u} \neq \underbar{0}$. On the
other hand, we also have $F(\underbar{u}+f(\underbar{u})) =
F(\underbar{u})+F(f(\underbar{u})),$
   with $F(\underbar{u}) \in U_1$ and $F(f(\underbar{u})) \in U_2$. Taking this fact into account,
since $V=U_1 \oplus U_2$, Equation (\ref{eq:colloineationdescribed})
implies that $F(\underbar{u})=\lambda_{\underbar{u}}\underbar{u}'$
and $F(f(\underbar{u}))=\lambda_{\underbar{u}}g(\underbar{u}')$.
Hence
\begin{equation}\label{eq:colloineationdescribed2}
F(f(\underbar{u}))=\lambda_{\underbar{u}}g\Big(\frac{F(\underbar{u})}{\lambda_{\underbar{u}}}\Big)=\frac{\lambda_{\underbar{u}}}{\lambda_{\underbar{u}}^{\sigma_g}}g(F(\underbar{u})).
\end{equation}

\noindent Let now $\underbar{u}$ and $\underbar{v}$ be two nonzero
vectors of $U_1$. If $\underbar{u}$ and $\underbar{v}$ are
$\F_{q^t}$-independent, by Equation
(\ref{eq:colloineationdescribed2}) we have
$$F(f(\underbar{u} +\underbar{v}
))=\lambda_{\underbar{u}+\underbar{v}}^{1-\sigma_g}g(F(\underbar{u}))+\lambda_{\underbar{u}+\underbar{v}}^{1-\sigma_g}g(F(\underbar{v})).$$
Also, we get
$$F(f(\underbar{u} +\underbar{v}
))=F(f(\underbar{u}))+F(f(\underbar{v}))=\lambda_{\underbar{u}}^{1-\sigma_g}g(F(\underbar{u}))+\lambda_{\underbar{v}}^{1-\sigma_g}g(F(\underbar{v})).$$
Hence
$\lambda_{\underbar{u}}^{1-\sigma_g}=\lambda_{\underbar{u}+\underbar{v}}^{1-\sigma_g}=\lambda_{\underbar{v}}^{1-\sigma_g}$,
which implies $\lambda_{\underbar{u}}/\lambda_{\underbar{v}}\in
\F_q$ since $Fix(\sigma_g)=\F_q$.

On the other hand, if $\underbar{u}$ and $\underbar{v}$ are
$\F_{q^t}$--dependent, choosing a vector $\underbar{w}\in U_1$, such
that $\underbar{w}\not\in \langle \underbar{u} \rangle_{q^t}$, and
arguing as above we get
$\lambda_{\underbar{u}}/\lambda_{\underbar{v}}\in\F_q$.

\noindent This means that for each $\underbar{u},\underbar{v}\in
U_1$ there exists an element
$\beta_{\underbar{u},\underbar{v}}\in\F_q$ such that
$\lambda_{\underbar{u}}=\beta_{\underbar{u},\underbar{v}}\lambda_{\underbar{v}}$.

\bigskip

\noindent Let $\underbar{u}\in U_1$, $\underbar{u}\neq
\underbar{0}$. Then, by Equation (\ref{eq:colloineationdescribed2}),
we get
\begin{equation}\label{eq:colloineationdescribed3} F(f(\alpha
\underbar{u})) =\alpha^{\sigma_f
\tau}\lambda_{\underbar{u}}^{1-\sigma_g}g(F(\underbar{u}))
\end{equation}for each $\alpha \in\Ft$.
 On the other hand, by Equation (\ref{eq:colloineationdescribed2}) again, for each  $\alpha \in \F_{q^t}$ we have
  \begin{equation}\label{eq:colloineationdescribed4}F(f(\alpha \underbar{u}))=\lambda_{\alpha\underbar{u}}^{1-\sigma_g}\alpha^{\tau \sigma_g}g(F(\underbar{u})),\end{equation}
where $\lambda_{\alpha\underbar{u}}\in \Ft$. Since
$\lambda_{\underbar{u}}/\lambda_{\alpha\underbar{u}}\in \F_q$, we
have
$\lambda_{\alpha\underbar{u}}^{1-\sigma_g}=\lambda_{\underbar{u}}^{1-\sigma_g}$.
Taking into account this fact, by Equations
(\ref{eq:colloineationdescribed3}) and
(\ref{eq:colloineationdescribed4}) we get
$\alpha^{\sigma_{f}\tau}=\alpha^{\tau \sigma_g}$ for each
$\alpha\in\Ft$, which implies $\sigma_f = \sigma_g$.

\noindent In the second case we have $F(U_1)=U_2$ and $F(U_2)=U_1$
and arguing as in the previous case we get $\sigma_g
=\sigma_f^{-1}$.

\medskip

Conversely, suppose that $\sigma_f = \sigma_g$ and let $\phi_F$ be
the collineation of $\Lambda$ defined by the  map $F$ of the
vector space $V=U_1 \oplus U_2$
 defined as follows $$F(\underbar{u}_1 + \underbar{u}_2)= \underbar{u}_1 +g(f^{-1}(\underbar{u}_2)),$$
 where $\underbar{u}_1 \in U_1$ and $\underbar{u}_2 \in U_2$. Then $\phi_F(L_{1,f})=L_{1,g}$. On the other hand, if $\sigma_f = \sigma_g^{-1}$, the collineation $\phi_F$ of $\Lambda$
 defined by the following map $F$ of $V=U_1 \oplus U_2$
$$F(\underbar{u}_1 + \underbar{u}_2)= g(\underbar{u}_1) + f^{-1}(\underbar{u}_2)$$
 sends $L_{1,f}$ to $L_{1,g}$. This concludes the proof.
\end{proof}

\noindent As a consequence of Theorem \ref{theorem:ProjEquiv} we
have the following

\begin{cor} \label{cor:ProjEquiva}
In the projective space $\Lambda=PG(2n-1,q^t)$ ($n\geq 2$, $t\geq
3$) there are $\varphi(t)/2$  orbits of scattered $\F_q$-linear
sets of $\Lambda$ of rank $tn$ of type $L_{\rho,f}$ under the
action of the collineation group of $\Lambda$.
\end{cor}
\begin{proof}
By the previous theorem, two linear sets $L_{\rho,f}$ and
$L_{\rho,g}$ are $P\Gamma L$-equivalent if and only if either
$\sigma_f = \sigma_g$ or $\sigma_f = \sigma_g^{-1}$. So the number
of orbits of such $\F_q$-linear sets under the action of $P\Gamma
L(2n,q^t)$ is $\chi/2$ where $\chi$ is the number of
$\F_q$-automorphisms $\sigma$ of $\Ft$ with $Fix(\sigma)=\F_q$. This
means that $\chi$ is the number of generators of the group
$Gal(\Ft:\F_q)$, i.e. $\chi=\varphi(t)$ is the number of positive
integers less than $t$ and coprime with $t$.
\end{proof}

\bigskip

In what follows we will show that each scattered $\F_q$-linear set
of pseudoregulus type can be obtained as in Theorem
\ref{thm:algebraicpseudoregulus}. Let's start by proving

\begin{theorem} \label{theorem:PSeuprojDesSpread}
Let $\Sigma \simeq PG(tn-1,q)$ be a subgeometry of
$\Sigma^*=PG(V,\F_{q^t})=PG(tn-1,q^t)$ defined by the  semilinear
collineation $\Psi$   of order $t$ of $\Sigma^*$. Also, let $\cD$ be
a Desarguesian $(t-1)$-spread of $\Sigma$ and denote by $\Theta$ a
director subspace of $\cD$. Then for each pair of integers $i_1,i_2
\in \{0,1,...,t-1\}$ such that $gcd(i_2-i_1,t)=1$, the linear set
obtained projecting $\Sigma$ from the subspace $\Gamma = \langle
\Theta^{\Psi^{i}} \,:\, i\neq i_1, i_2\rangle_{q^t}$ to $\Lambda =
\langle \Theta^{\Psi^{i_1}}, \Theta^{\Psi^{i_2}}\rangle_{q^t} $ is a
scattered $\F_q$-linear set of type $L_{\rho,f}$ described in
Theorem \ref{thm:algebraicpseudoregulus}.
\end{theorem}

\begin{proof}
Since $Fix (\Psi)=\Sigma$, the collineation $\Psi$ is induced by
an invertible semilinear map $g: V \longrightarrow V$ of
order $t>1$, with companion automorphism $\sigma$ such that
$Fix(\sigma)=\F_q$. Since $\Theta=PG(U,\F_{q^t})$ is  a director
subspace of the Desarguesian spread $\cD$, we have that
$\cD=\cD(\Theta)$ and $\Sigma^*=\langle
\Theta,\Theta^{\Psi},...,\Theta^{\Psi^{t-1}}\rangle_{q^t}$. Let
$i_1,i_2 \in \{0,1,...,t-1\}$ such that $gcd(i_2-i_1,t)=1$ and let
$\Gamma = \langle \Theta^{\Psi^{i}} \,:\, i\neq i_1,
i_2\rangle_{q^t}$ and $\Lambda = \langle \Theta^{\Psi^{i_1}},
\Theta^{\Psi^{i_2}}\rangle_{q^t} $. Then $dim \,\Gamma=n(t-2)-1$,
$dim\, \Lambda=2n-1$, $\Gamma \cap \Lambda =\Gamma\cap
\Sigma=\emptyset$ and hence we can project the subgeometry $\Sigma$
from the center $\Gamma$ to the axis $\Lambda$. By Theorem
\ref{Theor LP1}, the projection $L=p_{\Gamma}(\Sigma)$ is an
$\F_q$-linear set of $\Lambda$ of rank $tn$ and $\langle L
\rangle_{q^t} =\Lambda$.  Also, it is easy to see that
$$\Sigma=Fix(\Psi)=\{\langle   \underbar u +g(\underbar u)+g^2(\underbar u)+\dots+g^{t-1}(\underbar u) \rangle_{q^t} \, :\, \underbar u\in
U\setminus \{\underbar 0\}\},$$  and hence the projection of
$\Sigma$ from $\Gamma$ into $\Lambda$ is
$$L=p_{\Gamma}(\Sigma)=\{\langle g^{i_1}(\underbar u)+g^{i_2}(\underbar u)\rangle_{q^t} \, :\, \underbar u\in
U\setminus \{\underbar 0\}\}=  \{\langle \underbar v+f(\underbar
v)\rangle_{q^t} \, :\, \underbar v\in g^{i_1}(U)\setminus\{\underbar
0\}\},$$

\noindent where $f: \underbar v \in  g^{i_1}(U) \, \mapsto
g^{i_2-i_1}(\underbar v) \in  g^{i_2}(U)$. Since $f$ is an
invertible semilinear map whose companion authomorphism is
$\sigma^{i_2-i_1}$ and  $gcd(i_2-i_1,t)=1$, we have that
$Fix(\sigma^{i_2-i_1})=\F_q$. So, by Theorem
\ref{thm:algebraicpseudoregulus}, $L$ is a scattered  $\F_q$-linear
set of $\Lambda$ of pseudoregulus type with $\Theta^{\Psi^{i_1}}$
and $\Theta^{\Psi^{i_2}}$ as transversal spaces.
\end{proof}

\begin{remark}\rm \label{rema:gcd}
Note that, if $gcd(i_2-i_1,t)=s$, in the previous proof we have
$Fix(\sigma^{i_2-i_1})=\F_{q^s}$, and hence the linear set $L$
obtained projecting $\Sigma$ from $\Gamma = \langle
\Theta^{\Psi^{i}} \,:\, i\neq i_1, i_2\rangle_{q^t}$, is an
$\F_{q^s}$-linear set.
\end{remark}

\bigskip

\noindent Recall that, by Theorem \ref{Theor LP1}, every
$\F_q$-linear set $L$ of $\Lambda$ spanning the whole space  can be
obtained projecting a suitable subgeometry. If $L$ is of pseudoregulus type we can prove the
following

\begin{theorem}\label{theorem:DesargProj}
Put $\Lambda=PG(2n-1,q^t)$, $\Sigma^* = PG(tn-1,q^t)$ and $\Sigma
=Fix(\Psi)\simeq PG(tn-1,q)$ where $\Psi$ is a semilinear
collineation of $\Sigma^*$ of order $t$. Let $L$ be a scattered
$\F_q$--linear set of $\Lambda$ of pseudoregulus type with
associated pseudoregulus $\mathcal P$ obtained by projecting
$\Sigma$ into $\Lambda$ from an $(n(t-2)-1)$--dimensional subspace
$\Gamma$ disjoint from $\Sigma$. Then
\begin{itemize}
\item[(i)] the set
$$\mathcal{D}_L=\{\langle \Gamma, s\rangle_{q^t} \cap \Sigma \,:\, s
\in {\mathcal P}\}$$ is a  Desarguesian $(t-1)$--spread of $\Sigma$;
\item[(ii)] there exists a director space $\bar \Theta$ of
$\mathcal{D}_L$ such that
\begin{equation} \label{form:GammaGen}
\Gamma =\langle \bar \Theta,\bar \Theta^{\tau},\dots,\bar
\Theta^{\tau^{t-3}}\rangle_{q^t},
\end{equation}
 where
$\tau=\Psi^m$ with $gcd(m,t)=1$.
\end{itemize}
\end{theorem}
\begin{proof}
Since  each line $s$ of $\mathcal P$ has weight $t$ in $L$, it is
clear that $\langle \Gamma, s\rangle_{q^t} \cap \Sigma $ is a
$(t-1)$-dimensional subspace of $\Sigma$. Also, since the lines of
$\mathcal P$ are pairwise disjoint and $|\mathcal
P|=\frac{q^{nt}-1}{q^t-1}$, the set $\mathcal{D}_L$ in $(i)$ is a
$(t-1)$-spread of $\Sigma$.

 Denote by $T_1$
and $T_2$ the transversal spaces of $\mathcal{P}$ and  let $K_1$ be
the $(n(t-1)-1)$-dimensional subspace of $\Sigma^*$ joining $\Gamma$
and $T_1$. Since $L$ is disjoint from $T_1$, we have that $K_1 \cap
\Sigma = \emptyset$, and hence $K_1 \cap K_1^{\Psi} \cap \dots \cap
K_1^{\Psi^{t-1}} = \emptyset$. So $\Theta=K_1 \cap K_1^{\Psi} \cap
\dots \cap K_1^{\Psi^{t-2}}$ is an $(n-1)$-dimensional subspace of
$\Sigma^*$ and $\Sigma^*=\langle \Theta, \Theta^\Psi,\dots,
\Theta^{\Psi^{t-1}} \rangle_{q^t}$.

Now, for each line $s$ of $\mathcal P$, let $X_s=\langle \Gamma,
s\rangle_{q^t} \cap \Sigma$ be the corresponding element of the
spread $\mathcal{D}_L$ and denote by $X_s^*$ the $(t-1)$-dimensional
subspace of $\Sigma^*$  such that $X_s=X_s^* \cap \Sigma$, i.e.
$X_s^*$ is the $\F_{q^t}$--extension of $X_s$ in $\Sigma^*$. So
$X^*_s\subset \langle \Gamma,s\rangle_{q^t}$ and since $X^*_s$
intersects $\Sigma$ in a subspace of the same dimension, we have
that
 $(X_s^*)^\Psi=X_s^*$ (see \cite[Lemma 1]{Lu1999}). Also, let $P$ be the point $s\cap T_1$. Then
$\langle \Gamma, P \rangle_{q^t}$ is a hyperplane of $\langle
\Gamma, s\rangle_{q^t}$ and hence $H_s=\langle \Gamma,
P\rangle_{q^t} \cap X_s^*$ is a $(t-2)$-dimensional subspace of
$X_s^*$. Since $H_s\subseteq K_1$, we have that $H_s$ is disjoint
from $\Sigma$ and hence $H_s\cap H_s^\Psi\cap \dots\cap
H_s^{\Psi^{t-1}}=\emptyset$. So $H_s, H_s^\Psi, \dots,
H_s^{\Psi^{t-1}}$ are $t$ independent hyperplanes of $X_s^*$. This
implies that $ H_s\cap H_s^\Psi\cap \cdots\cap H_s^{\Psi^{t-2}}$ is
a point, say $R_s$,  of $X_s^*$. So
 $$R_s\in
X^*_s\cap (H_s\cap H_s^\Psi\cap \dots\cap H_s^{\Psi^{t-2}})
\subset X^*_s\cap (K_1\cap K_1^\Psi\cap \dots\cap
K_1^{\Psi^{t-2}})=X^*_s\cap \Theta$$
  for each $s \in \mathcal P$. By Remark \ref{rema:CharDesSpread}
  we get
that $\mathcal{D}_L$ is a Desarguesian spread of $\Sigma$ with
$\Theta$ as a director space. Also, $\Theta^{\Psi^i} \subset
K_1\,\,$ for each $i\neq 1$ and hence $K_1=\langle
\Theta^{\Psi^i}\,:\, i\neq 1\rangle_{q^t}$ and $K_1\cap
\Theta^\Psi=\emptyset$.

Similarly, if $K_2=\langle \Gamma, T_2\rangle_{q^t}$, we get that
$K_2 \cap K_2^{\Psi} \cap \cdots \cap K_2^{\Psi^{t-2}}$ is a
director space of the Desarguesian spread ${\cal D}_L$ and hence
there exists $m\in \{1,2,\dots,t-1\}$ such that $K_2 \cap K_2^{\Psi}
\cap \cdots \cap K_2^{\Psi^{t-2}}=\Theta^{\Psi^m}$ (see
\cite[Theorem 3]{LV}). So $\Theta^{\Psi^i} \subset K_2\,\,$ for each
$i\neq m+1$ and hence $K_2=\langle \Theta^{\Psi^i}\,:\, i\neq
m+1\rangle_{q^t}$ and $K_2\cap \Theta^{\Psi^{m+1}}=\emptyset$. This
means that

$$\Gamma=K_1\cap K_2=\langle
\Theta^{\Psi^i}\,:\, i\neq 1,m+1\rangle_{q^t}.$$ So, if
$\Psi^m=\tau$ and $\bar \Theta=\Theta^{\Psi^{2m+1}}$, we get
(\ref{form:GammaGen}) of $(ii)$. Finally, since $L$ is a scattered
$\F_q$-linear set, by Theorem \ref{theorem:PSeuprojDesSpread},
Remarks \ref{rema:gcd} and \ref{rema:sclinset}, we have that
$gcd(t,m)=1$.
 \end{proof}

By Theorems \ref{theorem:PSeuprojDesSpread} and
\ref{theorem:DesargProj} we have

\begin{theorem} \label{theorem:geomPseuAlegpseudo}
Each $\F_q$-linear set  of $PG(2n-1,q^t)$  of pseudoregulus type is
of the form $L_{\rho,f}$ described in Theorem
\ref{thm:algebraicpseudoregulus}.
\end{theorem}

\noindent Finally, by Theorem \ref{theorem:geomPseuAlegpseudo} and
by Corollary \ref{cor:ProjEquiva} we can state the following
classification result which generalizes \cite[Theorem 4]{LV}.

\begin{theorem} \label{thm:numbprojequipseudoregulus}
In the projective space $\Lambda=PG(2n-1,q^t)$ ($n\geq 2$, $t\geq
3$) there are $\varphi(t)/2$  orbits of maximum scattered
$\F_q$-linear sets of pseudoregulus type under the action of the
collineation group of $\Lambda$.
\end{theorem}

 \section{A class of maximum scattered $\F_q$-linear sets of
 $PG(1,q^t)$}\label{sec:scattered-line}

The arguments proving Theorem \ref{thm:algebraicpseudoregulus} can
be exploited to construct a class of maximum scattered $\F_q$-linear
sets of the projective line $\Lambda=PG(V,\F_{q^t})=PG(1,q^t)$ with
a structure resembling that of an $\F_q$-linear set
 of $PG(2n-1,q^t)$ $(n,t \geq 2)$ of pseudoregulus type. To this
aim let $P_1=\langle \underbar w \rangle_{q^t}$ and $P_2=\langle
\underbar v \rangle_{q^t}$ be two distinct points of $\Lambda$ and
let $\tau$ be an $\F_q$-automorphism of $\F_{q^t}$ such that
$Fix(\tau)=\F_q$; then for each $\rho \in \F_{q^t}^*$ the set
$$W_{\rho,\tau}=\{\lambda \underbar w +
\rho\lambda^{\tau}\underbar v  \,\, \colon \,\, \lambda \in
\F_{q^t}\},$$
 is an $\F_q$--vector subspace of $V$
of dimension $t$ and  $L_{\rho,\tau}:=L_{W_{\rho,\tau}}$ is a
scattered $\F_q$-linear set of $\Lambda$.
\begin{defin} \label{def:pseudoregulusline} \rm
We call the linear sets $L_{\rho,\tau}$  of {\it pseudoregulus type}
and we refer to the points $P_1$ and $P_2$ as {\it transversal
points} of $L_{\rho,\tau}$.
\end{defin}

If $L_{\rho,\tau} \cap L_{\rho',\tau} \neq \emptyset$, then
$L_{\rho,\tau}=L_{\rho',\tau}$. Note that
$L_{\rho,\tau}=L_{\rho',\tau}$ if and only if $N_{q^t/q}(\rho)=
N_{q^t/q}(\rho')$; so $P_1$, $P_2$ and the automorphism $\tau$
define a set of $q-1$ mutually disjoint maximum scattered linear
sets of pseudoregulus type admitting the same transversal points.
Such maximal scattered linear sets, together with $P_1$ and $P_2$,
cover the point set of the line $\Lambda=PG(1,q^t)$.

\begin{remark} \label{rema:baersubline} \rm
Since the group $PGL(2,q^t)$ acts 2--transitively on the points of
$\Lambda$, we may suppose that all $\F_q$--linear sets of
pseudoregulus type of $\Lambda$ have the same transversal points $P_1$ and $P_2$. This means that all such linear sets are only
determined by $\rho$ and by the automorphism $\tau$. Moreover, it is
easy to see that for each $\rho,\rho' \in \Ft^*$ the linear sets
$L_{\rho,\tau}$ and $L_{\rho',\tau}$ are equivalent. Indeed, it is
sufficient to consider the collineation of $\Lambda=PG(V,\F_{q^t})$
induced by the map $a\underbar w+b\underbar v\in V\mapsto
a\rho^{\tau^{-1}}\underbar w+b\rho'\underbar v\in V$. It follows
that, up to projectively equivalence, we may only consider $\F_q$--linear sets of type
$L_{1,\sigma_i}$, where $\sigma_i:\, x\in\F_{q^t}\mapsto
x^{q^i}\in\F_{q^t}$, with $i\in\{1,\dots,t-1\}$ and $gcd(i,t)=1$.
Now, by observing that, for each $i,j\in\{1,\dots,t-1\}$ with
$gcd(i,t)=gcd(j,t)=1$,
$$L_{1,\sigma_i}=\{\langle(x,x^{q^i})\rangle_{q^t}: x\in\F_{q^t}^*\}=\{\langle(1,x^{q^i-1})\rangle_{q^t}: x\in\F_{q^t}^*\}=$$ $$=\{\langle(1,a)\rangle_{q^t}: a\in\F_{q^t}^*, N_{q^t/q}(a)=1\}=\{\langle(x,x^{q^j})\rangle_{q^t}: x\in\F_{q^t}^*\}=L_{1,\sigma_j},$$
we have that in $\Lambda=PG(1,q^t)$ ($t\geq 2$) all $\F_q$-linear
sets of pseudoregulus type are equivalent to the linear set $L_{1,\sigma_1}$, under the action of the
collineation group of $\Lambda$. This result has been also proven in
\cite[Remark 2.2]{DoDuSub}.
\end{remark}

\begin{prop}\label{prop:unictraspoints}
If $L$ is an $\F_q$-linear set of pseudoregulus type of
$\Lambda=PG(1,q^t)$, $t\geq 3$, then its transversal points are
uniquely determined.
\end{prop}
\begin{proof}
By Remark \ref{rema:baersubline}, we may consider the $\F_q$--linear
set of pseudoregulus type
\begin{equation}\label{form:1}
L:=L_{1,\sigma_1}=\{\langle(\lambda,\lambda^q)\rangle_{q^t}:\,\lambda\in\F_{q^t}^*\},
\end{equation}
having $P_1=\langle(1,0)\rangle_{q^t}$ and $P_2=\langle(0,1)\rangle_{q^t}$ as transversal points.

Suppose that $L$ has another pair of transversal points
$P_1'=\langle\uw\rangle_{q^t}$ and $P_2'=\langle\uv\rangle_{q^t}$,
with $\uw=\langle(a,b)\rangle_{q^t}$ and
$\uv=\langle(c,d)\rangle_{q^t}$, such that $ad\ne bc$. Then $
L=\{\langle\eta
\uw+\rho\eta^\tau\uv\rangle_{q^t}:\,\eta\in\F_{q^t}^*\}, $ with
$\tau\in Aut(\F_{q^t})$. Moreover, arguing as in the previous
remark, we have that
\begin{equation}\label{form:2}
L=\{\langle
\uw+\rho\eta^{\tau-1}\uv\rangle_{q^t}:\,\eta\in\F_{q^t}^*\}=\{\langle
\uw+\rho\mu^{q-1}\uv\rangle_{q^t}:\,\mu\in\F_{q^t}^*\}=\{\langle
\mu\uw+\rho\mu^q\uv\rangle_{q^t}:\,\mu\in\F_{q^t}^*\}.
\end{equation}
By (\ref{form:1}) and (\ref{form:2}), we have that for each
$\lambda\in\F_{q^t}^*$, there exist
$\alpha_\lambda,\mu\in\F_{q^t}^*$ such that
$$
(\lambda,\lambda^q)=\alpha_\lambda(\mu\uw+\rho\mu^q\uv)=\alpha_\lambda(\mu
a+\rho\mu^qc,\mu b+\rho\mu^qd).
$$
Then, the above equality implies that
$$
\alpha_\lambda^{q-1}=\frac{\mu b+\rho\mu^qd}{(\mu
a+\rho\mu^qc)^q},$$ which gives $N_{q^t/q}(\mu
b+\rho\mu^qd)=N_{q^t/q}(\mu a+\rho\mu^qc)$ for each $\mu\in\F_{q^t}$, i.e.
\begin{equation}\label{form:5.1}
\prod_{i=0}^{t-1}(\mu^{q^i}
b^{q^i}+\rho^{q^i}\mu^{q^{i+1}}d^{q^i})=\prod_{i=0}^{t-1}(\mu^{q^i}
a^{q^i}+\rho^{q^i}\mu^{q^{i+1}}c^{q^i})\end{equation} for each
$\mu\in\F_{q^t}$. From the last equality we get a polynomial identity
in the variable $\mu$ of degree at most
$2q^{t-1}+q^{t-2}+\dots+q^3+q^2+q$. If $q\geq 3$, then $2q^{t-1}+q^{t-2}+\dots+q^3+q^2+q<q^t$, hence the polynomials in (\ref{form:5.1}) are the same. So comparing the coefficients of the terms of maximum
degree, we get
\begin{equation}\label{form:6}
d^{1+q+q^2+\dots +q^{t-2}}b^{q^{t-1}}=c^{1+q+q^2+\dots
+q^{t-2}}a^{q^{t-1}}.
\end{equation}
Also, comparing the coefficients of the terms of degree
$2q^{t-1}+q^{t-2}+\dots+q^3+q^2+1$, for $t>2$, we have
\begin{equation}\label{form:7}
d^{q+q^2+\dots +q^{t-2}}b^{q^{t-1}}b=c^{q+q^2+\dots
+q^{t-2}}a^{q^{t-1}}a.
\end{equation}
If $bd\ne 0$, then $ac\ne 0$ and dividing both sides of Equations (\ref{form:6}) and
(\ref{form:7}), we get $\frac db=\frac ca$, a contradiction since
$P_1'\ne P_2'$. If $b=0$, from (\ref{form:6}) we have $c=0$ and
hence $P_1'=P_1$ and $P_2'=P_2$; if $d=0$, then also $a=0$ by
(\ref{form:6}) and hence $P_1'=P_2$ and $P_2'=P_1$.

\medskip

If $q=2$, reducing (\ref{form:5.1}) modulo $\mu^{q^t}-\mu$, we get that the two polynomials of (\ref{form:5.1}) have degree at
most $q^{t-1}+q^{t-2}+\dots+q^3+q^2+q+1$. So, comparing the
coefficients of the terms of degree
$q^{t-1}+q^{t-2}+\dots+q^3+q^2+2$, and of the terms of degree
$q^{t-2}+q^{t-3}+\dots+q^3+q^2+2$ (for $t>2$), and arguing as above
we get the same result. This completes the proof.
\end{proof}

\begin{remark}
{\rm Note that if $t=2$, then $L_{\rho,\tau}$ is a Baer subline of
$\Lambda=PG(1,q^2)$ and $P_1$ and $P_2$ are conjugated with respect
to the semilinear involution of $\Lambda$ fixing $L_{\rho,\tau}$
pointwise. Hence, in such a case, the transversal points are not
uniquely determined.}
\end{remark}

\begin{remark}\label{rem:examplenonpseudoregulustype}\rm
Let $L_{\rho,f}$ be an $\F_q$-linear set of pseudoregulus type of
$PG(2n-1,q^t),$ $n>1$, and let $\mathcal P_{L_{\rho,f}}$ be the
associated $\F_q$-pseudoregulus. By (\ref{generalizedpseudoregulus})
and Definition \ref{def:pseudoregulusline}, we observe that for each
line $s\in {\cal P}_{\rho,f}$, the set $L_{\rho,f}\cap s$ is a
linear set of pseudoregulus type whose transversal points are the
intersections of $s$ with the transversal subspaces of
$\mathcal{P}_{L_{\rho,f}}$.
\end{remark}

We conclude this section by giving some  examples of maximum
scattered $\F_q$-linear sets  which are not of pseudoregulus type.

\begin{example}\label{example:noPseud}\rm
$(i)$\quad Let $$L_\rho=\{\langle (x,\rho
x^{q}+x^{q^{t-1}})\rangle_{q^t} : x\in\F_{q^t}^*\},$$ where
$\rho\in\F_{q^t}$ such that $N_{q^t/q}(\rho)\neq 1$. By
\cite[Theorem 2]{LP2001} $L_\rho$ is a scattered $\F_q$--linear set
of rank $t$. Moreover, if $q>3$, $\rho\ne 0$ and $t\geq 4$, by
\cite[Theorem 3]{LP2001}, there is no collineation of $PG(1,q^t)$
mapping $L_\rho$ to $L_{1,\sigma_1}$. Hence, by Remark
\ref{rema:baersubline}, $L_\rho$ is a maximum scattered
$\F_q$--linear set which is not of pseudoregulus type when $q>3$.

\noindent $(ii)$\quad Let $$L=\{\langle (x_0,x_1,\dots,x_{n-1},\rho
x_0^{q}+x_0^{q^{t-1}},x_1^q,\dots,x_{n-1}^q)\rangle_{q^t}:\,
x_i\in\F_{q^t}\},$$ with $\rho\in\F_{q^t}^*$ and $N_{q^t/q}(\rho)\ne
1$. It is easy to see that $L$ is a scattered $\F_q$--linear set of
rank $tn$. Also, the line $r$ with equations
$x_1=x_2=\dots=x_{n-1}=0$ is a line of weight $t$ in $L$ and, by the
previous arguments   $r\cap L$ is an $\F_q$--linear set which is not
of pseudoregulus type for $q>3$. So by Remark
\ref{rem:examplenonpseudoregulustype} and by Point $(i)$,  for each
$q>3$, $t\geq 4$ and $n\geq 2$, $L$ is not of pseudoregulus type.
\end{example}

\section{Linear sets and the variety $\Omega(\cS_{n,n})$}\label{sec:SegreVariety}

Let $\M=\M(n,q)~(n\geq 2)$ be the vector space of the matrices of
order $n\times n$ with entries in $\F_q$ and let
$PG(n^2-1,q)=PG(\M,\F_q).$ The Segre variety
$\cS_{n,n}=\cS_{n,n}(q)$ of $PG(n^2-1,q)$ is the set of all points
$\langle X \rangle_{q}$ of $PG(n^2-1,q)$ such that $X$ is a matrix
of rank $1$. Here below we list some well known properties of such a
variety, that can be found in \cite[pp. 98--99]{Harris},
\cite{Herzer} and \cite[Section 25.5]{HiTh1991}. Precisely,

\begin{itemize}
\item $|\cS_{n,n}|=(\frac{q^n-1}{q-1})^2$;
\item maximal subspaces of $\cS_{n,n}$ have dimension $n-1$;
\item there are two families $\cR_1$ and $\cR_2$ of maximal subspaces of $\cS_{n,n}$, which are the {\it systems} of $\cS_{n,n}$. Spaces of the same system are pairwise
skew and any two spaces of different systems meet in exactly one
point. The elements of each system partition $\cS_{n,n}$. Moreover,
$|\cR_1|=|\cR_2|=\frac{q^n-1}{q-1}$;
\item the automorphism group $Aut(\cS_{n,n})$ of $\cS_{n,n}$ is isomorphic to $P\Gamma L(n,q)\times P\Gamma L(n,q)\times C_2$, and it is the group of all
collineations of $PG(n^2-1,q)$ fixing or interchanging the two systems of $\cS_{n,n}.$
\end{itemize}

\noindent A $k$--dimensional subspace $S$ of $PG(n^2-1,q)$ is a {\em
$k$--th secant subspace} to $\cS_{n,n}$ when $S=\langle
P_1,P_2,\dots,P_{k+1}\rangle_{q}$ and $\{P_1,P_2, \dots,
P_{k+1}\}\subset S \cap \cS_{n,n}.$ The $(n-2)$--th {\em secant
variety} $\Omega(\cS_{n,n})$ of $\cS_{n,n}$ is the set of all points
of $PG(n^2-1,q)$ belonging to an $(n-2)$--th secant subspace to
$\cS_{n,n}.$ Note that
\begin{equation}\label{form:omega}
\Omega(\cS_{n,n})=\{\langle X \rangle_{q} \mid X \in
\M(n,q)\setminus \{{\bf 0}\}, det X=0\},
\end{equation}
i.e. $\Omega(\cS_{n,n})$ is the algebraic variety, also called {\it
determinantal hypersurface}, defined by the non--invertible matrices of $\M(n,q)$.

\bigskip

Regarding $\F_{q^n}$ as an $n$--dimensional vector space over $\F_q$
and fixing an $\F_q$--basis $\cal B$ of $\F_{q^n}$, each matrix $M$ of $\M=\M(n,q)$ defines an $\F_q$--endomorphism $\varphi_M$ of $\F_{q^n}$, and conversely.
The map $\phi_M:\,M\in\M\mapsto \varphi_M\in\E$, where $\E=End(\F_{q^n},\F_q)$ is the
$n^2$--dimensional vector space  of all the
$\F_q$--endomorphisms of $\F_{q^n}$, is an isomorphism between the vector spaces $\M$ and $\E$. By using such an isomorphism, we have that the elements of $\E$ with rank 1 define in
$PG(\E,\F_q)=PG(n^2-1,q)$ the Segre variety $\cS_{n,n}$. Recalling that each element $\varphi\in\E$ can be written as $\varphi(x)=\sum_{i=0}^{n-1}\beta_ix^{q^i}$, with $\beta_i\in\F_{q^n}$, we get the following result.
\begin{prop}\label{prop:segre-variety}
Let $\P=PG(\E,\F_q)=PG(n^2-1,q)$ and let $\cS_{n,n}$ be the Serge
variety of $\P$ defined by the elements of $\E$ with rank 1. Then
$$\cS_{n,n}=\{\langle t_\lambda\circ Tr\circ t_\mu\rangle_{q}:\
\lambda,\mu\in\F_{q^n}^*\}\ \ ^(\footnote{$\circ$ stands for
composition of maps}^),$$ where $t_\alpha:x\in\F_{q^n}\mapsto \alpha
x\in\F_{q^n}$, with $\alpha\in\F_{q^n}$ and $Tr:x\in\F_{q^n}\mapsto
x+x^q+\dots+x^{q^{n-1}}\in\F_{q}$. Moreover $\cR_1=\{X(\lambda):
\lambda\in\F_{q^n}^*\}$ and $\cR_2=\{X'(\lambda):
\lambda\in\F_{q^n}^*\}$, where
$$X(\lambda)=\{\langle t_\alpha\circ Tr\circ t_\lambda\rangle_{q}:\
\alpha\in\F_{q^n}^*\}\mbox{\quad and \quad} X'(\lambda)=\{\langle
t_\lambda\circ Tr\circ t_\alpha\rangle_{q}:\
\alpha\in\F_{q^n}^*\},$$ are the two systems of $\cS_{n,n}$.
Finally, $\Omega(\cS_{n,n})$ is defined by the non--invertible elements of
$\mathbb E$.
\end{prop}
\begin{proof}
Note that, for each $\lambda,\mu\in\F_{q^n}^*$, we have
$$ker\, (t_\lambda\circ
Tr\circ t_\mu)=\frac 1\mu\, ker\,Tr,$$  so $dim\,(ker\,
(t_\lambda\circ Tr\circ t_\mu))= n-1$ (i.e., $t_\lambda\circ Tr\circ
t_\mu$ is an element of $\E$ of rank 1) and hence $\langle
t_\lambda\circ Tr\circ t_\mu\rangle_q\in\cS_{n,n}$. Also, for each
$\lambda',\mu'\in\F_{q^n}^*$, $t_\lambda\circ Tr\circ
t_\mu=t_{\lambda'}\circ Tr\circ t_{\mu'}$ if and only if
$\frac\lambda{\lambda'}=\frac\mu{\mu'}\in\F_q^*$. Then direct
computations show that $|\{\langle t_\lambda\circ Tr\circ
t_\mu\rangle_{q}:\
\lambda,\mu\in\F_{q^n}^*\}|=(\frac{q^n-1}{q-1})^2$, and hence
$\cS_{n,n}=\{\langle t_\lambda\circ Tr\circ t_\mu\rangle_{q}:\
\lambda,\mu\in\F_{q^n}^*\}$ $^(\footnote{Alternatively, by \cite[Thm
2.24]{LN} it can be easily seen that the maps $t_\lambda\circ
Tr\circ t_\mu$ are all the $\F_q$--endomorphisms of $\F_{q^n}$ with
rank 1.}^)$.

Also, it is easy to prove that for each $\lambda\in\F_{q^n}^*$, the
sets $X(\lambda)$ and $X'(\lambda)$ are $(n-1)$--dimensional
subspaces of $\P$ contained in $\cS_{n,n}$. Moreover, for each
$\lambda,\mu\in\F_{q^n}^*$, two subspaces $X(\lambda)$ and $X(\mu)$
are either disjoint or equal, and this latter case holds true if and
only if $\frac{\lambda}{\mu}\in\F_q^*$. The same happens for
$X'(\lambda)$ and $X'(\mu)$. This implies that
$|\{X(\lambda): \lambda\in\F_{q^n}^*\}|=|\{X'(\lambda):
\lambda\in\F_{q^n}^*\}| =\frac{q^n-1}{q-1}$. Also,
$X(\lambda)\cap X'(\mu)=\{\langle t_\mu\circ Tr \circ t_\lambda
\rangle_q\}$ is a point. Then $\cR_1=\{X(\lambda):
\lambda\in\F_{q^n}^*\}$ and $\cR_2=\{X'(\lambda):
\lambda\in\F_{q^n}^*\}$ are the systems of $\cS_{n,n}$. Finally, by (\ref{form:omega}) the last
part of the assertion follows.
\end{proof}

\bigskip

For each $\varphi\in\E$, where
$\varphi(x)=\sum_{i=0}^{n-1}\beta_ix^{q^i}$, the {\it conjugate}
$\overline\varphi$ of $\varphi$ is defined by
$\overline{\varphi}(x)=\sum_{i=0}^{n-1}\beta_i^{q^{n-i}}x^{q^{n-i}}$.
Precisely, $\bar\varphi$ is the {\em adjoint} map of  $\varphi$ with
respect to the  non--degenerate bilinear form of $\F_{q^n}$
\begin{equation}\label{form:polarity}
\beta(x,y)=Tr_{q^n/q}(xy).
\end{equation}

\noindent The map
$$
T \colon \varphi \in \E \mapsto \overline{\varphi} \in \E,
$$
is an involutory $\F_q$--linear permutation of $\E$ and
straightforward computations show that
\begin{eqnarray}
&&\overline{\varphi\circ\psi}=\overline{\psi}\circ
\overline{\varphi},\quad\quad
\overline{\varphi^{-1}}=(\overline{\varphi})^{-1}\quad\quad\mbox{
for each $\varphi,\psi\in\E$};\label{form:mapT}\\
&&\overline{t_\lambda}=t_\lambda\quad\quad\mbox{for each
$\lambda\in\F_{q^n}.$}\label{form1:mapT}
\end{eqnarray}
Moreover, it can be easily checked that
$ker\,\varphi=(Im\,\overline{\varphi})^\perp$, where $\perp$ is
the polarity defined by (\ref{form:polarity}), and hence
$dim\,(ker\,\varphi)=dim\,(ker\,\overline{\varphi})$. Then $T$
induces in $\P$ a linear involutory collineation $\Phi_T$
preserving the varieties $\cS_{n,n}$ and $\Omega(\cS_{n,n})$ and
interchanging the systems $\cR_1$ and $\cR_2$ of $\cS_{n,n}$.
Indeed, we have
\begin{equation}\label{form::mapTsyst}
X(\mu)^{\Phi_T}=X'(\mu)\quad\quad\mbox{for each $\mu\in\F_{q^n}^*$}.
\end{equation}
The subgroup $H(\cS_{n,n})$ of $P\Gamma L(n^2,q)$ fixing the systems
$\cR_1$ and $\cR_2$ of $\cS_{n,n}$ is isomorphic to $P\Gamma
L(n,q)\times P\Gamma L(n,q)$, and such a group has index 2 in the
group $Aut(\cS_{n,n})=Aut(\Omega(\cS_{n,n}))\simeq P\Gamma L(n,q)\times P\Gamma L(n,q)\times C_2$ (\cite[Thm. 3]{La2008}). Hence
$Aut(\cS_{n,n})=\langle H(\cS_{n,n}),\Phi_T\rangle$.

\bigskip

\bigskip

\noindent Let $\cI:=\{\langle t_\lambda\rangle_{q}:\, \lambda\in\F_{q^n}^*\}$.
Then $\cI$ is an $(n-1)$--dimensional subspace of $\P$ disjoint from
the variety $\Omega(\cS_{n,n})$ and
$$\cD_1(\cI)=\{\{\langle t_\alpha\circ \varphi\rangle_{q}:\, \alpha\in\F_{q^n}^*\}:\
\varphi\in\E\setminus\{{\mathbf 0}\}\}$$ and
$$\cD_2(\cI)=\{\{\langle \varphi\circ t_\alpha\rangle_{q}:\, \alpha\in\F_{q^n}^*\}:\
\varphi\in\E\setminus\{{\mathbf 0}\}\}$$ are two Desarguesian
spreads of $\P$ (see, e.g., \cite[Example 3 and Theorem
14]{BaLu2011}) such that

\begin{itemize}
\item[$I_1)$] $\cI\in\cD_i(\cI)$ and $\cR_i\subset \cD_i(\cI)$, for each $i\in\{1,2\}$.
\end{itemize}
Also, we explicitly note that
\begin{itemize}
\item[$I_2)$] $\Phi_T$ fixes $\cI$ pointwise and, by (\ref{form::mapTsyst}), $\cD_1(\cI)^{\Phi_T}=\cD_2(\cI)$.
\end{itemize}

\noindent Let $\Pi_{n-1}(\cD_1(\cI))$ be the $\F_q$--linear representation
of the projective space $PG(n-1,q^n)$ defined by the Desarguesian
spread $\cD_1(\cI)$ of $\P$. Let $\Upsilon_1$ be the linear
collineation of $\P$ defined as
$$\Upsilon_1:\ \langle \varphi\rangle_{q}\in\P\mapsto \langle\varphi'\rangle_{q}\in\P,$$
where $\varphi'(x)=\sum_{i=0}^{n-1}a_{i-1}^qx^{q^i}$ if
$\varphi(x)=\sum_{i=0}^{n-1}a_ix^{q^i}$, taking the indices $i$
modulo $n$.
\begin{itemize}
\item[$I_3)$] The collineation $\Upsilon_1$ fixes the Desarguesian spread $\cD_1(\cI)$ and
induces a collineation $\bar \Upsilon_1$ in
$\Pi_{n-1}({\cD_1(\cI)})$ of order $n$ whose fixed point set
consists of the elements of $\cR_1$. Hence, $\cR_1$ turns out to be
a subgeometry of $\Pi_{n-1}({\cD_1(\cI)})$ isomorphic to
$PG(n-1,q)$.
\end{itemize}
We explicitly note that
\begin{equation}\label{form:spannedI}
\cI^{\Upsilon_1^j}=\{\langle x\mapsto \lambda x^{q^j}\rangle_{q}:\
\lambda\in\F_{q^n}^*\} \mbox{\quad and \quad}\langle
\cI,\cI^{\Upsilon_1},\dots,\cI^{\Upsilon_1^{n-1}}\rangle_{q}=\P.
\end{equation}
So $\cI$, in $\Pi_{n-1}({\cD_1(\cI)})$, is a point whose orbit under
the action of the cyclic group $\langle\bar\Upsilon_1\rangle$
has maximum size $n$.

\noindent In the same way,
\begin{itemize}
\item[$I_4)$] the collineation $\Upsilon_2=\Phi_T^{-1}\circ
\Upsilon_1\circ\Phi_T$ fixes the Desarguesian spread $\cD_2(\cI)$
and induces a collineation $\bar \Upsilon_2$ in
$\Pi_{n-1}({\cD_2(\cI)})$ of order $n$ whose set of fixed points
consists of the elements of $\cR_2$.
\end{itemize}
Also,
\begin{itemize}
\item[$I_5)$]
$\cI^{\Upsilon_2^i}=\cI^{\Upsilon_1^{n-i}}$.
\end{itemize}

\bigskip

\noindent Let $\cO_\cI$ be the orbit, under the action of the group
$H(\cS_{n,n})$, of the $(n-1)$--dimensional subspace $\cI$ of $\P$.
A subspace belonging to this orbit will be called a {\it
$\cD$--subspace} of $\P$. In the following we will study the
geometric properties of the $\cD$--subspaces of $\P$ under the
action of $H(\cS_{n,n})$.

\medskip

\begin{theorem}\label{thm:Dsubspace}
Let $X$ be a $\cD$--subspace of $\P=PG(\E,\F_q)=PG(n^2-1,q)$, then
there exist two Desarguesian spreads $\cD_1(X)$ and $\cD_2(X)$ of
$\P$ such that:
\begin{itemize}
\item [$D_1)$] $X\in\cD_i(X)$ and $\cR_i\subset \cD_i(X)$  for each $i=1,2$,
 \item [$D_2)$]  there is a semilinear collineation $\bar\Xi_i$ of $\Pi_{n-1}(\cD_i(X))$ of order $n$
induced by a linear collineation $\Xi_i$ of $\P$ fixing
the Desarguesian spread $\cD_i(X)$. Moreover, $\cR_i=Fix\,\bar\Xi_i$
is a subgeometry of $\Pi_{n-1}(\cD_i(X))$ isomorphic to a $PG(n-1,q)$.
\end{itemize}
Also, there exists an involutory collineation $\Phi$ of $\P$ such
that
\begin{itemize}
\item [$D_3)$] $\Phi$ fixes $X$ pointwise,
\item [$D_4)$] $\cD_1(X)^\Phi=\cD_2(X)$.
\end{itemize}
\end{theorem}

\begin{proof}
Let $g$ be an element of $H(\cS_{n,n})$ such that $\cI^g=X$. By
$I_1)$, $\cD_i(X):=\cD_i(\cI)^g$, for each $i\in\{1,2\}$, is a
Desarguesian spread of $\P$ containing $X$ and the system $\cR_i$,
i.e. $D_1)$ is satisfied. Putting $\Xi_i:=g\circ \Upsilon_i\circ
g^{-1}$ and $\Phi:=g\circ \Phi_T\circ g^{-1}$ and taking $I_2)$,
$I_3)$ and $I_4)$ into account, $D_2)$, $D_3)$ and $D_4)$ follow.
\end{proof}

\noindent This allows us to give the following

\begin{defin}\label{def:conjugate}{\rm
Let $X$ be a $\cD$--subspace of $\P$ and let $\Xi_i$ ($i\in\{1,2\}$)
be one of the two collineations of $\P$ described in $D_2)$. Each of
the $\cD$--subspaces $X^{\Xi_i^j}$, with $j\in\{0,1,\dots,n-1\}$, is
said to be a {\em conjugate} of $X$. Note that, by $I_5)$,
$X^{\Xi_2^j}=X^{\Xi_1^{n-j}}$.}
\end{defin}

\begin{remark}  \label{rem:polarline}
{\rm If $n=2$, then $\E=End(\F_{q^2},\F_q)$ and $\cS_{2,2}$ is the
hyperbolic quadric $Q^+(3,q)$ of $\P=PG(\E,\F_q)=PG(3,q)$ defined by
the quadratic form
$$\varphi\in\E\mapsto a^{q+1}-b^{q+1}\in\F_q,$$ where
$\varphi(x)=ax+bx^q$. Hence, the group $H(\cS_{2,2})$ is the
subgroup of the orthogonal group $P\Gamma O^+(4,q)$ fixing the
reguli of $Q^+(3,q)$. Also, the $H(\cS_{2,2})$--orbit of the line
$\cI$, is the set of all external lines to the quadric. Moreover,
the involutory linear collineation $\Upsilon_1$ of $\P$ described
above is
$$\langle x\mapsto ax+bx^q\rangle_{q}\longmapsto\langle x\mapsto b^qx+a^qx^q\rangle_{q}.$$
This means that the conjugate of $\cI$ is the line
$\cI^{\Upsilon_1}=\{\langle x\mapsto \mu x^q\rangle_{q}:\,
\mu\in\F_{q^2}^*\}$, which is the polar line of $\cI$ with respect
to the quadric $Q^+(3,q)$.}
\end{remark}

\subsection{Linear sets and presemifields}

A finite semifield is a finite division algebra which is not
necessarily associative and throughout this paper the term semifield
will always be used to denote a finite semifield (see, e.g.,
\cite{LaPo201*} Chapter 6 for definitions and notations on finite
semifields).  Every field is a semifield and  the term {\it proper
semifield} means a semifield which is not a field. The left nucleus
$\N_l$ and the center  $\K$ of a semifield $\bS$ are fields
contained in $\bS$ as substructures ($\K$ subfield of $\N_l$) and
$\bS$ is a vector space over $\N_l$ and over $\K$. Semifields are
studied up to an equivalence relation called {\it isotopy} and the
dimensions of a semifield over its left nucleus and over its center
are invariant up to isotopy.

Let $\bS$ be a semifield with center $\K$ and left nucleus $\N_l$
and  let $(\F_q,\Ft)$ be a pair of fields such that $\F_q\leq \K$
and  $\Ft \leq \N_l$; then $\bS$ is a finite extension of $\Ft$ and
hence it has size $q^{nt}$ for some integer $n\geq 1$. If $\bS$ is a
proper semifield, then $n\geq 2$. Also, up to isotopy, we may assume that $\bS=(\F_{q^{nt}},+,\star)$, where $$x\star y=\varphi_y(x)$$ with $\varphi_y\in\E=End(\F_{q^{nt}}, \Ft)$.  The set
$$\cC_\bS=\{\varphi_y\,:\, x\in \F_{q^{nt}}\, \mapsto\, x\star y\in \F_{q^{nt}}\,|\,
y\in \F_{q^{nt}}\} \subset \E$$ is the {\it semifield spread set}
associated with $\bS$ ({\it spread set} for short): $\cC_\bS$ is an
$\F_q$-subspace of $\E$ of rank $nt$ and each non-zero element of
$\cC_\bS$ is invertible. Hence, for each pair $(\F_q,\Ft)\subseteq
(\K, \N_l)$, we can associate with $\bS$ the $\F_q$-linear set of
rank $nt$ of the projective space $\P=PG(\E,\Ft)=PG(n^2-1,q^t)$
defined by the non-zero elements of $\cC_\bS$. Such a linear set
turns out to be disjoint from the variety $\Omega(\cS_{n,n}(q^t))$
of $\P$ defined by the non-invertible elements of $\E$. Isotopic
semifields produce in $\P=PG(n^2-1,q^t)$ linear sets which are
equivalent with respect to the action of the group
$H(\cS_{n,n}(q^t))$, and conversely (see \cite{L2003} for $n=2$ and
\cite{La} for  $n\geq 2$). Among all the pairs $(\F_q,F_{q^t})$ such
that $\F_q\subseteq \K$ and $\F_{q^t}\subseteq\N_l$, the pair $(\K,
\N_l)$ has the following properties: $(i)$ maximizes the field of
linearity of the linear set associated with $\bS$, $(ii)$ minimizes
the dimension of the projective space $\P$ in which the linear set
is embedded and $(iii)$ minimizes the group $H(\cS_{n,n})$. For
instance, if $\bS=\F_q$, then $\N_l=\K=\F_q$ and hence the linear
set associated with the field $\F_q$, with respect to the pair
$(\F_q,\F_q)$, is the point $PG(\F_q,\F_q)$; whereas, if $\F_{q'}$
is a subfield of $\F_q$, $q=q'^n$, then the linear set associated
with $\F_q$, with respect to the pair $(\F_{q'},\F_{q'})$, is an
$(n-1)$-dimensional subspace of $\P=PG(n^2-1,q')$ disjoint from the
variety $\Omega(\cS_{n,n}(q'))$ of $\P$, which is a $\cD$-subspace
of $\P$, and conversely (see \cite[Theorem 20]{La2008}). In what
follows, we will call the linear set associated with $\bS$ with
respect to the maximum pair $(\K,\N_l)$, the {\it relevant} linear
set associated with $\bS$.

In the next sections we will characterize, up to the action of the
group $H({\mathcal S}_{n,n})$, the relevant linear sets associated
with some classical semifields: the Generalized Twisted Fields and
the  Knuth semifields 2-dimensional over their left nucleus.

\subsection{Generalized Twisted Fields}
If $\bS$ satisfies all the axioms for a semifield except,
possibly, the existence of the identity element for the
multiplication, then it is a {\it presemifield}.  In such a case
the nuclei and the center of $\bS$ are defined as fields of linear
maps contained   in $End(\bS,\F_p)$ (where $p$ is the
characteristic of $\bS$) (see, e.g., \cite[Theorem 2.2]{MP}) and
all that we stated and defined above for semifields
can be applied to presemifields.\\

 The Generalized Twisted Fields are presemifields
constructed by A.A. Albert in \cite{albert61bis}.  By \cite[Lemma
1]{albert61} a Generalized Twisted Field $\G$ with center of order
$q$, $n$-dimensional over its left nucleus ($n\geq 2$) and
$tn$-dimensional over its center is of type
$\mathbb{G}=(\F_{q^{nt}},+,\star)$ ($q=p^e$, $p$ prime) with

\begin{equation} \label{multi:GTF}
x\star y = yx - cy^{q^m}x^{q^{tl}},
\end{equation}
\noindent
 where $c \in \F_{q^{nt}}^*$, $c \neq
x^{q^{tl}-1}y^{q^m-1}$ for every $x,y \in \F_{q^{nt}}$, and  $1\leq
l\leq n-1$, $1\leq m\leq nt-1$, $m\neq tl$. Since we required
$dim_\K\G=nt$ and $dim_{\N_l}\G=n$, we also have
 $gcd(l,n)=gcd(t,m)=1$.
From the previous conditions we get  $q>2$ and, if $t=1$, then
$n\geq 3$.  In terms of linear maps,  by \cite[Theorem 2.2]{MP} and
by \cite[Lemma 1]{albert61} we can describe the left nucleus and the
center of $\G$ as follows
$$\N_l=\{t_\lambda: \ x\in  \F_{q^{nt}} \,\mapsto \,
\lambda x\in \F_{q^{nt}}\,|\, \lambda \in \Ft\}\subset \E,$$
$$\K=\{t_\lambda: \ x\in  \F_{q^{nt}} \,\mapsto \,  \lambda
x\in \F_{q^{nt}}\,|\, \lambda \in \F_q\}\subset \E,$$ where
$\E=End(\F_{q^{nt}}, \Ft)=V(n^2,\Ft)$. The {\it spread set}
associated with $\G$ is
$$\C=\{\varphi_y: x\in \F_{q^{nt}} \, \mapsto \, x\star y \in \F_{q^{nt}}\, | \, y\in
\F_{q^{nt}}\}\subset \E$$ and it is an $\F_q$-subspace of $\E$ of
dimension $nt$. Hence $\C$ defines an $\F_q$-linear set of rank $nt$
in the projective space $\P=PG(\E,\Ft)=PG(n^2-1,q^t)$; precisely
$$L(\G)=L_\C=\{\langle \varphi_y \rangle_{q^t} \,:\, y\in
\F_{q^{nt}}^*\},$$ which is the relevant linear set associated with
$\G$. Since the nonzero elements of $\C$ are invertible, $L(\G)$ is
disjoint from the variety $\Omega(\cS_{n,n}(q^t))$ of $\P$ defined
by the non-invertible elements of $\E$. By (\ref{multi:GTF}) it is
clear that $L(\G)$ is contained in the subspace
$\Lambda=PG(2n-1,q^t)$ of $\P$ joining the $\cD$-space
$\cI=\{\langle t_\lambda\rangle_{q^t}:\, \lambda\in\F_{q^{nt}}^*\}$
and its conjugate $\cI^{\Upsilon_1^l}=\{\langle x\mapsto \lambda
x^{q^{tl}}\rangle_{q^t}:\, \lambda\in\F_{q^{nt}}^*\}$, precisely
$$\Lambda=\{\langle x \mapsto Ax+Bx^{q^{lt}}\rangle_{q^t} \,:\, A,B\in
\F_{q^{nt}}\}.$$ Note that $\Lambda$ defines a line $PG(1,q^{nt})$
in the $\Ft$-linear representation $\Pi_{n-1}(\cD_1(\cI))$. Also,
since $gcd(t,m)=1$, it is easy to verify that, if $t\geq 2$, then
$L(\G)$ is a maximum scattered $\F_q$-linear set of $\Lambda$ and,
hence, $\Lambda=\langle L(\G) \rangle_{q^t}$.

\begin{prop}\label{prop:linsetGenTwFi}
Let $\mathbb{G}=(\F_{q^{nt}},+,\star)$ be a Generalized Twisted
Field $n$-dimensional over its left nucleus and $tn$-dimensional
over its center. Let $\P=PG(\E,\Ft)=PG(n^2-1,q^t)$ (where
$\E=End(\F_{q^{nt}}, \Ft)$), $\Lambda=\{\langle x \mapsto
Ax+Bx^{q^{lt}}\rangle_{q^t} \,:\, A,B\in \F_{q^{nt}}\}$ and
$\Pi_{n-1}(\cD_1(\cI))$ be the $\F_q$--linear representation of
$PG(n-1,q^{n})$.\\
{\bf (a)} If $t=1$, then  $(a.\,i)$ $L(\G)$ is an $(n-1)$-dimensional subspace
of $\P=PG(n^2-1,q)$  contained in $\Lambda$ and  in the linear
representation $\Pi_{n-1}(\cD_1(\cI))\simeq PG(n-1,q^{n})$; $(a.\,ii)$ $L(\G)$
induces an $\Ft$-linear set of pseudoregulus type with transversal points $\cI$ and  $\cI^{\Upsilon_1^l}$.  \\
{\bf (b)} If $t\geq 2$, then $L(\G)$ is a scattered  $\F_q$-linear
set of rank $tn$ of pseudoregulus
type of $\Lambda$ with $\cI$ and  $\cI^{\Upsilon_1^l}$  as transversal spaces.\\
\end{prop}
\begin{proof}
{\bf (a)} If $t=1$, then $\C$ is an $\F_q$-subspace of
$\E=End(\F_{q^{n}}, \F_q)$; i.e. $L(\G)$ is  just an
$(n-1)$-dimensional subspace of $\P$ contained in $\Lambda$. Note
that the map
$$\Phi: \{\langle t_\lambda\circ \varphi\rangle_{q}:\, \lambda\in\F_{q^n}^*\} \in \cD_1(\cI) \, \mapsto \, \langle(a_0,a_1,\dots,a_{n-1})\rangle_{q^n} \in PG(n-1,q^n)$$
where $\varphi(x)=\sum_{i=0}^{n-1}a_ix^{q^i}$ is a linear
collineation between $\Pi_{n-1}(\cD_1(\cI))$ and $PG(n-1,q^n)$ such
that $\cI^\Phi=\langle(1,0,\dots,0)\rangle_{q^n}$,
$\cI^{{\Upsilon_1^l}\Phi}=\langle(0,\dots,1,\dots,0)\rangle_{q^n}$
and $\Lambda^\Phi$ is the line of $PG(n-1,q^n)$ with equations
$x_i=0$ for $i\neq 0,l$. Also $L(\G)^\Phi=\{\langle
(y,0,0,\dots,-cy^{q^m},0,\dots,0)\rangle_{q^n} \,:\, y\in
\F_{q^{n}}^*\}$ is an $\F_q$-linear set contained in the line
$\Lambda^\Phi$ of $PG(n-1,q^{n})$. By Definition
\ref{def:pseudoregulusline}, $L(\G)^\Phi$ is a maximum scattered
$\F_q$-linear set of pseudoregulus  type of $\Lambda^\Phi$, with
transversal points $\cI^\Phi$ and
$\cI^{\Upsilon_1^l \Phi}$. This proves {\bf (a)}.\\
{\bf (b)} If $t\geq 2$, then  the collineation

$$\Phi_f : \langle t_y: x \mapsto yx\rangle_{q^t} \in \cI \,\mapsto \, \langle f(t_y): x \mapsto-cy^{q^m}x^{q^{lt}}\rangle_{q^t} \in
\cI^{\Upsilon_1^l},$$ \noindent
 is a semilinear collineation between $\cI$
and $\cI^{\Upsilon_1^l}$
 with companion automorphism $\sigma:
\alpha \in \Ft \mapsto\alpha^{q^m}\in \Ft$ and, since $gcd(t,m)=1$,
$Fix (\sigma)=\F_q$. Hence, by Theorem
\ref{thm:algebraicpseudoregulus}, $L_{W_{1,f}}$ is an $\F_q$-linear
set of $\Lambda$ of pseudoregulus type with transversal spaces $\cI$
and $\cI^{\Upsilon_1^l}$, and since

$$W_{1,f}=\{t_y+f(t_y) \,:\, y\in \F_{q^{nt}}\}=\{x\, \mapsto\, yx-cy^{q^m}x^{q^{tl}}\,|\, y\in \F_{q^{nt}}
\}=\C,$$ \noindent Case {\bf (b)} follows.
\end{proof}

Now, we will prove that the properties of $L(\G)$ described in
 Proposition
\ref{prop:linsetGenTwFi} completely characterize, up to isotopy, the
Generalized Twisted Fields.

\begin{theorem} \label{theorem:caraGenTwFie}
Let $\bS$ be a presemifield of order $q^{nt}$ with
 $\F_q$ contained in its center and $\Ft$ contained in its left nucleus and let $L(\bS)$ be the
associated linear set with respect to the pair $(\F_q,\Ft)$. Also,
assume that $L(\bS)$ is contained in a $(2n-1)$-dimensional subspace
of $\P=PG(n^2-1,q^t)$ joining two conjugated $\cD$-spaces $X$ and $X'$ of $\P$.  Suppose that either Case (a) or Case (b) below holds:\\
{\bf (a)} $t=1$ and  $L(\bS)$ induces, in the linear representation
$\Pi_{n-1}(\cD_1(X))\simeq PG(n-1,q^{n})$, an $\F_q$-linear set of
pseudoregulus type of the line $PG(1,q^n)$ of $\Pi_{n-1}(\cD_1(X))$
joining the points $X$ and  $X'$, with transversal points $X$ and  $X'$;  \\
{\bf (b)} $t\geq 2$ and  $L(\bS)$ is a maximum scattered
$\F_q$-linear set  of pseudoregulus type of $\langle
X,X'\rangle_{q^t}$ with $X$ and  $X'$  as transversal spaces; \\then
$\bS$ is isotopic to a Generalized Twisted Field.
\end{theorem}
\begin{proof} Without loss of generality we may assume that
$\bS=(\F_{q^{nt}},+,\ast)$ with $\F_q$ contained in $\K$ and $\Ft$
contained in $\N_l$. Let $\E=End(\F_{q^{nt}},\F_q)$ and let
$\C=\{\varphi_y: x\in \F_{q^{nt}} \, \mapsto\, x\ast y \in
\F_{q^{nt}}\, | \, y\in \F_{q^{nt}}\}$ be the spread set defined by
$\bS$. Note that describing $\bS$ corresponds, up to isotopy, to
describing the associated linear set $L(\bS)=L_\C$ in the projective
space $\P=PG(\E,\Ft)=PG(n^2-1,q^t)$, up to the action of the group
$H(\cS_{n,n}(q^t))$. Since all the $\cD$-spaces of $\P$ belong to
the same $H(\cS_{n,n}(q^t))$-orbit, we may assume, up to isotopy,
that $X=\cI$, so
$X'=\cI^{\Upsilon_1^l}$ for some $l\in \{1,\dots,n-1\}$ (see (\ref{form:spannedI})). \\
{\bf (a)} By Definition \ref{def:pseudoregulusline} $L(\bS)=L_{\rho,
\tau}$ where $\rho\in \F_{q^n}^*$ and $\tau: x \mapsto x^{q^m}$ is
an automorphism of $\F_{q^n}$ such that $gcd(m,n)=1$. This implies
that
$$\C=\{\varphi_y: x\in \F_{q^{n}}
\, \mapsto\, xy+\rho y^{q^m}x^{q^{l}} \in \F_{q^{n}}\, | \, y\in
\F_{q^{n}}\}.$$ Hence $x \ast y=xy-cy^{q^m}x^{q^{l}}$ where
$c=-\rho$, i.e. $\bS$, up to isotopy, is a Generalized Twisted
Field.
 \noindent {\bf (b)} By Theorem
\ref{theorem:geomPseuAlegpseudo}, $L(\bS)$ is of type $L_{\rho,f}$
with transversal spaces $\cI$ and $\cI^{\Upsilon_1^l}$. Hence,
there exist a semilinear collineation
$$\Phi_f: \langle t_y \rangle_{q^t} \in \cI \,\mapsto \langle f(t_y) \rangle_{q^t} \in \cI^{\Upsilon_1^l} $$
with companion automorphism $\sigma\in Aut(\F_{q^t})$ such that
$Fix(\sigma)=\F_q$ and an element $\rho \in \F_{q^{nt}}^*$ such that
$$\C=\{t_y+\rho f(t_y)\,:\, y\in \F_{q^{nt}}\}.$$
This implies that
$$f(t_y): x\mapsto\eta y^{q^m}x^{q^{tl}},$$
where $\eta \in \F_{q^{nt}}^*$, $1\leq m\leq nt-1$ and
$gcd(t,m)=1$. Hence, putting $c=-\eta\rho$, we have
$$\C=\{\varphi_y: x\in \F_{q^{nt}} \, \mapsto\,
xy-cy^{q^m}x^{q^{tl}} \in \F_{q^{nt}}\, | \, y\in \F_{q^{nt}}\},$$
this means that  $x\ast y=xy-cy^{q^m}x^{q^{tl}}$, and hence $\bS$,
up to isotopy, is a Generalized Twisted Field.
\end{proof}

\bigskip
\noindent Note that if $n=2$ then $t\geq 2$ and by Remark \ref{rem:polarline},
we can restate Theorem \ref{theorem:caraGenTwFie} as follows, which
is a generalization of \cite[Theorems 4.3, 3.7]{CPT} and
\cite[Theorems 4.12, 4.13]{MPT}.

\begin{cor}\label{cor:caract2dimGTF}
Let $\bS$ be a presemifield of order $q^{2t}$ with center $\F_q$ and
left nucleus $\F_{q^t}$. If $L(\bS)$ is an $\F_q$--linear set of
$PG(3,q^t)$ of pseudoregulus type with transversal lines external to
the quadric $\cS_{2,2}=Q^+(3,q^t)$ pairwise polar with respect to
the polarity defined by $Q^+(3,q^t)$, then $\bS$ is isotopic to a
Generalized Twisted Field.
\end{cor}

\subsection{2--dimensional Knuth Semifields}
The Knuth semifields 2--dimensional over the left nucleus and
$2t$--dimensional ($t\geq 2$) over the center $\F_q$ are the
following:

$\K_{17}=(\F_{q^t}\times \F_{q^t} ,+,\ast)$ and
$\K_{19}=(\F_{q^t}\times \F_{q^t} ,+,\star)$ (see \cite[p. 241
(Multiplications $(17)$ and $(19)$)]{d}), with
$$(u,v)\ast (x,y) = (u,v)\begin{pmatrix} x & y\\
fy^{\sigma} & x^{\sigma}+y^{\sigma}g\end{pmatrix}$$ and
$$(u,v)\star (x,y) = (u,v)\begin{pmatrix} x & y\\
fy^{\sigma^{-1}} & x^{\sigma}+yg\end{pmatrix},$$ where $\sigma\in
Aut(\F_{q^t})$, $Fix\,\sigma=\F_q$, and $f$ and $g$ are non--zero
elements in $\F_{q^t}$ such that the polynomial $x^{q+1}+gx-f$ has
no root in $\F_{q^t}$.

\noindent The spread sets (of matrices) associated with $\K_{17}$
and $\K_{19}$ are
$$\C_{17}=\Big\{\begin{pmatrix} x & y\\
fy^{\sigma} &
x^{\sigma}+y^{\sigma}g\end{pmatrix}\,:\,x,y\in\F_{q^t}\Big\}\subset
\M$$ and
$$\C_{19}=\Big\{\begin{pmatrix} x & y\\
fy^{\sigma^{-1}} &
x^{\sigma}+yg\end{pmatrix}\,:\,x,y\in\F_{q^t}\Big\}\subset \M,$$
respectively, where $\M=\M(2,q^t)$ is the vector space of the
$2\times 2$--matrices over $\F_{q^t}$.

The sets $\C_{17}$ and $\C_{19}$ are $\F_q$-subspaces of $\M$ of
dimension $2t$ and hence they define  $\F_q$-linear sets of rank
$2t$ in the projective space $\P=PG(\M,\Ft)=PG(3,q^t)$. Precisely,
using the coordinatization $\begin{pmatrix} x_0 & x_1\\
x_2 & x_3\end{pmatrix}\mapsto (x_0,x_1,x_2,x_3)$,
\begin{equation}\label{LK1}
L(\K_{17})=\{\langle(x,y,fy^{\sigma},x^{\sigma}+gy^{\sigma})\rangle_{q^t}\colon
x,y\in \F_{q^t},\,(x,y)\ne(0,0)\}\end{equation} and
\begin{equation}\label{LK1}
L(\K_{19})=\{\langle(x,y,fy^{\sigma^{-1}},x^{\sigma}+gy)\rangle_{q^t}\colon
x,y\in \F_{q^t},\,(x,y)\ne(0,0)\}\end{equation} are the relevant
linear sets associated with the semifields $\K_{17}$ and $\K_{19}$,
respectively.

\noindent Recall that $L(\K_{17})$ and $L(\K_{19})$ are disjoint
from the hyperbolic quadric $Q^+(3,q^t)$ of $\P$ defined by the
non--invertible matrices of $\M$. Let $\cR_1$ be the regulus of
$Q^+(3,q^t)$ containing the line $x_2=x_3=0$ and let $\cR_2$ be the
opposite one.

\begin{remark}\label{rem:transpKnuth}
\rm Note that the collineation $\Phi_T$ of $\P$ defined by the
transpose operation on matrices fixes the quadric $Q^+(3,q^t)$ and
interchanges the reguli $\cR_1$ and $\cR_2$ and
$\Phi_T(L(\K_{17}(\sigma,f,g)))=L(\K_{19}(\sigma,\frac
1{f^{\sigma^{-1}}},\frac gf))$. In other words, the family $\K_{19}$
is the transpose family of $\K_{17}$ (see \cite[Section 5]{knuth})
\end{remark}

\begin{prop}\label{prop:knuthlinearset}
(1) $L(\K_{17})$ is an $\F_q$--linear set of $\P=PG(3,q^t)$ of
pseudoregulus type, whose transversal lines belong to $\cR_1$.

(2) $L(\K_{19})$ is an $\F_q$--linear set of $\P=PG(3,q^t)$ of
pseudoregulus type, whose transversal lines belong to $\cR_2$.
\end{prop}
\begin{proof}
Let $r$ and $r'$ be the lines of $\cR_1$ with equations $r:\,
x_2=x_3=0$ and $r':\, x_0=x_1=0$. Then the map $$f:\
(x,y,0,0)\mapsto (0,0,f y^\sigma, x^\sigma+gy^\sigma)$$ induces a semilinear collineation $\Phi_f$ between $r$ and $r'$
having $\sigma$ as a companion automorphism. Then by Theorem
\ref{thm:algebraicpseudoregulus}, $L_{1,f}$ is an $\F_q$-linear set
of pseudoregulus type. Since $L_{1,f}=L(\K_{17})$ we get $(1)$.

\noindent Case $(2)$ follows from Remark \ref{rem:transpKnuth} and Case
$(1)$.
\end{proof}

In the next theorem we prove that the descriptions of $L(\K_{17})$
and $L(\K_{19})$ given in Proposition \ref{prop:knuthlinearset}
characterize the semifields $\K_{17}$ and $\K_{19}$ up to isotopism,
generalizing some results contained in \cite{CPT} and \cite{MPT} for
$t=2$.

\begin{theorem} \label{theorem:caraKnuth}
Let $\bS$ be a presemifield of order $q^{2t}$ with
 $\F_q$ contained in its center and $\Ft$ contained in its left nucleus and let $L(\bS)$ be the
associated linear set with respect to the pair $(\F_q,\Ft)$. If
$L(\bS)$ is an $\F_q$-linear set  of pseudoregulus type of
$\P=PG(3,q^t)$ with associated transversal lines $r$ and $r'$
contained in $Q^+(3,q^t)$, then $\bS$ is isotopic to a Knuth
semifield $\K_{17}$ or $\K_{19}$. Precisely, if $r,r'\in\cR_1$, then
$\bS$ is isotopic to a semifield $\K_{17}$, whereas, if
$r,r'\in\cR_2$, then $\bS$ is isotopic to a semifield $\K_{19}$.
\end{theorem}
\begin{proof} Without loss of generality, we may assume that $\bS=(\F_{q^t}\times
\F_{q^t},+,\circ)$, with $\F_{q^t}\times \{0\}$ contained in its
left nucleus. This implies that
$$(u,v)\circ (x,y)=(u,v)M,$$ where $M=M_{x,y}\in \M$. So, the
spread set of matrices associated with $\bS$ is
$$\C=\Big\{M_{x,y}=\left(
                          \begin{array}{cc}
                            m_0(x,y) & m_1(x,y) \\
                            m_2(x,y) & m_3(x,y) \\
                          \end{array}
                        \right)
:\ x,y\in\F_{q^t}\Big\}$$ and
$$L_{\C}=L(\bS)=\{\langle(m_0(x,y),m_1(x,y),m_2(x,y),m_3(x,y))\rangle_{q^t}\colon
x,y\in \F_{q^t},\,(x,y)\ne(0,0)\},$$ where $m_i(x,y)$ are
$\F_q$--linear maps. Assume that the transversal lines $r$ and $r'$
of $L(\bS)$ are contained in $\cR_1$. Since the group
$H(\cS_{2,2})=G$ (see Remark \ref{rem:polarline}) acts
$2$--transitively on the lines of $\cR_1$, we can suppose that
$r=\{\langle(x_0,x_1,0,0)\rangle_{q^t}\colon x_0,x_1\in
\F_{q^t},\,(x,y)\ne(0,0)\}$ and
$r'=\{\langle(0,0,x_2,x_3)\rangle_{q^t}\colon x_2,x_3\in
\F_{q^t},\,(x,y)\ne(0,0)\}.$ Note that the stabilizer $G_{\{r,r'\}}$
in the group $G$ of the lines $r$ and $r'$ acts transitively on the
points of $r$. If $P$ is any point of $r$, then the stabilizer
$G_{\{r,r',P\}}$ of $P$ in $G_{\{r,r'\}}$ fixes the point
$P^{\perp}\cap r'$ and acts transitively on the remaining points of
$r'$. This means that we can suppose, without loss of generality,
that the line $s$ with equations $x_1=x_2=0$ belongs to the
pseudoregulus associated with $L(\bS)$. Let $R = r\cap s=\langle
(1,0,0,0)\rangle_{q^t}$ and $R'=r'\cap s
=\langle(0,0,0,1)\rangle_{q^t}$. By Theorems
\ref{thm:numbprojequipseudoregulus} and
\ref{thm:algebraicpseudoregulus} there exist a semilinear
collineation $\Phi:\ \langle(x,y,0,0)\rangle_{q^t}\in r\mapsto
\langle(0,0,h(x,y),g(x,y))\rangle_{q^t}\in r'$ having $\sigma\in
Aut(\F_{q^t})$, with $Fix\,\sigma=\F_q$, as companion automorphism,
and an element $\rho\in\F_{q^t}^*$ such that
$$L(\bS)=\{\langle(x,y,\rho h(x,y),\rho g(x,y))\rangle_{q^t}\colon
x,y\in \F_{q^t},\,(x,y)\ne(0,0)\}.$$ Since $\Phi$ is semilinear with
companion automorphism $\sigma$, we have that
$$h(x,y)=a_1x^\sigma+a_2y^\sigma\mbox{\quad and \quad}
g(x,y)=b_1x^\sigma+b_2y^\sigma,$$ where
$a_1,a_2,b_1,b_2\in\F_{q^t}$. Also, since the line $s$ belongs to
the pseudoregulus associated with $L(\bS)$, we have $\Phi(R)=R'$,
i.e. $h(1,0)=a_1=0$. So $L(\bS)=\{\langle(x,y,\alpha y^\sigma,\beta
x^\sigma+\gamma y^\sigma)\rangle_{q^t}\colon x,y\in
\F_{q^t},\,(x,y)\ne(0,0)\}$, where $\alpha=\rho a_2$, $\beta=\rho
b_1$, $\gamma=\rho b_2$. Note that $L(\bS)\cap Q^+(3,q^t)=\emptyset$
implies $\beta \neq 0$. The collineation $\omega$ of $\P$ defined as
$\omega:\langle(x_0,x_1,x_2,x_3)\rangle_{q^t}\mapsto
\langle(x_0,x_1,x_2/\beta,x_3/\beta)\rangle_{q^t}$ fixes the reguli
of $Q^+(3,q^t)$ and
$$L(\bS)^\omega=\{\langle(x,y,fy^\sigma,
x^\sigma+gy^\sigma)\rangle_{q^t}\colon x,y\in
\F_{q^t},\,(x,y)\ne(0,0)\},$$ where $f=\frac\alpha\beta$ and
$g=\frac\gamma\beta$, i.e., up to isotopy, $\bS$ is a $\K_{17}$
Knuth semifield. Finally, if the transversal lines of $L(\bS)$
belong to $\cR_2$, arguing as in the previous case, we get that
$\bS$ is isotopic to a $\K_{19}$ Knuth semifield.
\end{proof}
\bigskip

 \subsection*{Acknowledgement}
We thank the referees for their valuable comments; these have
increased the readability of the article. Also, we wish to thank G.
Donati and N. Durante for their helpful suggestions regarding Remark
\ref{rema:baersubline}.




\bigskip

\bigskip

\bigskip

\noindent
\begin{tabular}{lll}
G. Lunardon and R. Trombetti & & G. Marino and O. Polverino\\
Dip. di Matematica e Applicazioni\ \  & & Dip. di Matematica e Fisica  \\
Universit\`a di Napoli ``Federico II'' & & Seconda Universit\`a degli Studi di Napoli \\
80126 Napoli, Italy & & 81100 Caserta, Italy\\
{\em lunardon@unina.it}, {\em rtrombet@unina.it} & & {\em giuseppe.marino@unina2.it}, {\em olga.polverino@unina2.it}\\
\end{tabular}


\begin{thebibliography}{10}

\bibitem{albert61}
{\sc A.A. ~Albert}: Isotopy for Generalized Twisted Fields, {\it An.
Acad. Brasil. Ci.}, {\bf 33} (1961), 265--275.

\bibitem{albert61bis}
{\sc A.A. ~Albert}: Generalized Twisted Fields, {\it Pacific J.
Math.}, {\bf 11} (1961), 1--8.

\bibitem{BaLu2011}
{\sc  L. Bader, G. Lunardon}: Desarguesian Spreads, {\it Ric.
Mat.}, {\bf 60} (2011), 15--37.

\bibitem{BL}
{\sc A. Blokhuis, M. Lavraw}: Scattered spaces with respect to a
spread in $PG(n,q)$, {\it Geom. Dedicata}, {\bf 81} no. 1-3 (2000),
231--243.

\bibitem{CPT}
{\sc I. Cardinali, O. Polverino, R. Trombetti}: Semifield planes
of order $q^4$ with kernel $\F_{q^2}$ and center $\F_q$, {\it
Europ. J. Combin.}, {\bf 27} (2006), 940-961.

\bibitem{d}
{\sc P. ~Dembowski}: {\it Finite Geometries}, Springer Verlag,
Berlin, 1968.

\bibitem{DoDuSub}
{\sc G. Donati, N. Durante}: Scattered linear sets generated by
collineations between pencils of lines, submitted.

\bibitem{Fr} {\sc J.W. Freeman}: Reguli and pseudo-reguli in
$PG(3,s^2)$, {\it Geom. Dedicata}, {\bf 9} (1980), 267--280.


\bibitem{Harris} {\sc J. Harris}: {\it Algebraic Geometry. A first course}.
Springer Verlag, New York, 1992.

\bibitem{Herzer} {\sc A., Herzer}: Generalized Segre varieties, {\it Rend.
Mat.}, {\bf 7} (1986), 1–-36.

\bibitem{HiTh1991} {\sc J.W.P. Hirschfeld, J.A. Thas}: {\it General Galois geometries}.
Oxford University Press, Oxford, 1991.


\bibitem{knuth}
{\sc D.E. Knuth}: Finite Semifields and Projective Planes, {\it
Journal of Algebra}, {\bf 2} (1965), 182--217.

\bibitem{La2008}{\sc M. Lavrauw}: On the isotopism classes of finite semifields, {\it Finite Fields and Their Applications}, {\bf 14} (2008),
897–-910.


\bibitem{La}{\sc M. Lavrauw}: Finite semifields with a large nucleus and higher secant varieties to Segre varieties, {\it Adv. Geom.}, {\bf 11} (2011), no. 3,
399–-410.


\bibitem{LMPT} {\sc M. Lavrauw, G. Marino, O. Polverino, R. Trombetti}: $\F_q$--pseudoreguli of $PG(3,q^3)$ and scattered semifields of order $q^6$,
{\it  Finite Fields Appl.}, {\bf 17} (2011), 225--239.


\bibitem{LaPo201*} {\sc M. Lavrauw, O. Polverino}: \textit{Finite semifields}. Chapter 6 in Current
research topics in Galois Geometry (J. De Be Storme, Eds.), NOVA
Academic Publishers, Pub. Date 2011, ISBN: 978-1-61209-523-3.



\bibitem{LaVanVo2010} {\sc M. Lavrauw, G. Van de Voorde}: On linear
sets on a projective line, {\it Des. Codes Cryptogr.}, {\bf 56}
(2010), 89--104.


\bibitem{LV} {\sc M. Lavrauw, G. Van de Voorde}: Scattered linear
sets and pseudoreguli, {\it The Electronic Journal of Comb.}, {\bf 20}(1) (2013).


\bibitem{LN} {\sc R. Lidl, H. Niederreiter}: {\it Finite fields}, Encyclopedia Math. Appl., Vol. 20,
Addison-Wesley, Reading (1983) (Now distributed by Cambridge
University Press).

\bibitem{Lu1999} {\sc G. Lunardon}: Normal spreads, {\it Geom. Dedicata}, {\bf 75} (1999), 245--261.

\bibitem{L2003}{\sc G. Lunardon}:
Translation ovoids, {\it J. Geom.}, {\bf 76} (2003), 200--215.

\bibitem{LP2001} {\sc G. Lunardon, O. Polverino}:
Blocking sets and derivable partial spreads, {\it J. Algebr.
Comb.}, {\bf 14} (2001), 49--56.


\bibitem{LP} {\sc G. Lunardon, O. Polverino}:
Translation ovoids of orthogonal polar spaces, {\it Forum Math.},
{\bf 16} (2004), 663--669.

\bibitem{MP}{\sc G. Marino, O. Polverino}: {\it On the nuclei of a finite semifield}. Theory and applications of finite fields, Contemp. Math., {\bf 579}, Amer. Math. Soc., Providence, RI (2012), 123–-141.



\bibitem{MPT}{\sc G. Marino, O. Polverino, R. Trombetti}: On ${\mathbb
F}_q$--linear sets of $\PG(3,q^3)$ and semifields,  {\em J. Combin.
Theory Ser. A}, {\bf 114} (2007), 769--788.


\bibitem{Polverino2010} {\sc O. Polverino}: {Linear sets in Finite Projective
Spaces}, {\it Discrete Math.}, {\bf 310} (2010), 3096--3107.


\bibitem{Segre} {\sc B. Segre}: Teoria di Galois, fibrazioni
proiettive e geometrie non desarguesiane, {\it Ann. Mat. Pura
Appl.}, {\bf 64} (1964), 1--76.

\end{thebibliography}
\end{document}